\documentclass[11pt,a4paper]{article}
\usepackage[utf8]{inputenc}
\usepackage{microtype}
\usepackage{amsmath,amsthm,amssymb}
\usepackage{mathrsfs} 
\usepackage[T1]{fontenc}
\usepackage{ae,aecompl}
\usepackage{times}
\usepackage[british]{babel}
\usepackage{aliascnt}
\usepackage[colorlinks,pagebackref,hypertexnames=false]{hyperref} 
\usepackage{url}
\usepackage[alphabetic,backrefs]{amsrefs}
\usepackage{nth}
\usepackage{float}
\usepackage[all]{xy}
\usepackage{tikz-cd} 
\usepackage{bookmark}
\usepackage[margin=1.9cm]{geometry}
\usepackage{enumerate}

\newcommand{\rG}{{\rm G}}

\newcommand{\rU}{{\rm U}}

\newcommand{\cH}{\mathcal{H}}
\newcommand{\cI}{\mathcal{I}}

\newcommand{\cV}{\mathcal{V}}

\newcommand{\sX}{\mathscr{X}}

\newcommand{\fh}{{\mathfrak h}}

\newcommand{\fm}{{\mathfrak m}}

\newcommand{\R}{\mathbb{R}}
\newcommand{\C}{\mathbb{C}}

\newcommand{\SO}{{\rm SO}}
\newcommand{\Sp}{{\rm Sp}}
\newcommand{\Spin}{{\rm Spin}}
\newcommand{\SU}{{\rm SU}}

\newcommand{\U}{{\rm U}}

\renewcommand{\det}{\mathop\mathrm{det}\nolimits}
\newcommand{\Diff}{\mathrm{Diff}}

\renewcommand{\epsilon}{\varepsilon}

\newcommand{\Ric}{{\rm Ric}}

\newcommand{\Stab}{\mathrm{Stab}}

\newcommand{\tr}{\mathop{\mathrm{tr}}\nolimits}

\newcommand{\vol}{\mathrm{vol}}

\newcommand{\qandq}{\quad\text{and}\quad}
\newcommand{\qforq}{\quad\text{for}\quad}
\newcommand{\qwithq}{\quad\text{with}\quad}

\def\<{\mathopen{}\left<}
\def\>{\right>\mathclose{}}
\def\({\mathopen{}\left(}
\def\){\right)\mathclose{}}

\usepackage{multicol, color}

\definecolor{gold}{rgb}{0.85,.66,0}
\definecolor{cherry}{rgb}{0.9,.1,.2}
\definecolor{burgundy}{rgb}{0.8,.2,.2}
\definecolor{orangered}{rgb}{0.85,.3,0}
\definecolor{orange}{rgb}{0.85,.4,0}
\definecolor{olive}{rgb}{.45,.4,0}
\definecolor{lime}{rgb}{.6,.9,0}
\definecolor{green}{rgb}{.2,.7,0}
\definecolor{grey}{rgb}{.4,.4,.2}
\definecolor{brown}{rgb}{.4,.3,.1}

\newtheorem{theorem}{Theorem}[section]
\newtheorem{proposition}[theorem]{Proposition}
\newtheorem{corollary}[theorem]{Corollary}
\newtheorem{lemma}[theorem]{Lemma}

\numberwithin{equation}{section}

\theoremstyle{remark}
\newtheorem{remark}{Remark}[theorem]

\theoremstyle{definition}
\newtheorem{definition}[theorem]{Definition}

\def\al{\alpha}

\def\w{\wedge}
\def\R{\mathbb{R}}
\def\C{\mathbb{C}}

\def\Lm{\Lambda}

\def\om{\omega}
\def\Om{\Omega}
\def\vp{\varphi}
\def\ip{\raise1pt\hbox{\large$\lrcorner$}} 

\usepackage{mathtools}

\usepackage[colorinlistoftodos,prependcaption, textsize=tiny, textwidth=2.5cm,color=green]{todonotes}
\usepackage{comment}

\def\ufm{\underline{\mathfrak{m}}}

\allowdisplaybreaks

\usepackage{titling}
\usepackage{changepage}

\begin{document}
	\title{Harmonic flow of quaternion-Kähler structures}
	\author{
	Udhav Fowdar
    \qquad\qquad
    Henrique N. S\'a Earp 
    \\
  \emph{University of Campinas (Unicamp)}}
  \date{}

\maketitle
	
\vspace{-0.5cm}	
\begin{abstract}
We formulate the gradient Dirichlet flow of $\Sp(2)\Sp(1)$-structures on $8$-manifolds, as the first systematic study of a geometric quaternion-Kähler (QK) flow. Its critical condition of \emph{harmonicity} is especially relevant in the QK setting, since torsion-free structures are often topologically obstructed. We show that the conformally parallel property implies harmonicity, extending a result of Grigorian in the $\rG_2$ case. We also draw several comparisons with $\Spin(7)$-structures.

Analysing the QK harmonic flow, we prove an almost-monotonicity formula, which implies to long-time existence under small initial energy, via $\epsilon$-regularity. We set up a theory of harmonic QK solitons, constructing a non-trivial steady example. We produce explicit long-time solutions: one, converging to a torsion-free limit on the hyperbolic plane; and another,  converging to a limit which is harmonic but not torsion-free, on the manifold $\SU(3)$. We also study compactness and the formation of singularities.

\end{abstract}
	
\begin{adjustwidth}{0.95cm}{0.95cm}
    \tableofcontents
\end{adjustwidth}

\newpage
\section{Introduction}
In this paper we study the harmonic flow of $H$-structures introduced in \cite{Loubeau2019} for $H=\Sp(2)\Sp(1)$ on $8$-manifolds. We refer to it simply as the  quaternion-K\"ahler harmonic flow. The corresponding flows for $H=\mathrm{G}_2,\Spin(7),\rU(n)$ have been studied quite extensively in recent years cf. \cites{Dwivedi2019, Grigorian2019, Dwivedi2021, HeLi2021}. To the best of our knowledge, this is the first study in the literature of a geometric flow of quaternion-K\"ahler structures. 

Harmonic structures arise naturally as the critical points of the $L^2$-energy of the intrinsic torsion of an $H$-structure (wih $H \subset \SO(n)$) and as such can be interpreted as the `best' representative $H$-structure in a given isometric class. The harmonic flow is precisely the negative gradient flow of this energy functional. 
A well-known result of Poon-Salamon in \cite{PoonSalamon1991} asserts that there are only three compact (torsion-free) quaternion-K\"ahler $8$-manifolds, and these are all symmetric spaces. Thus, harmonic structures can be viewed as the next most special objects on any other compact $8$-manifolds admitting $\Sp(2)\Sp(1)$-structures. 
This article provides the first step towards a more general study of geometric flows of  $\Sp(n)\Sp(1)$-structures on $4n$-manifolds, and it can be read as a detailed instance of the abstract theory simultaneously formulated in \cite{Fadel2022}.
One long-term prospect of $\Sp(n)\Sp(1)$-flows would be an analytic approach to the LeBrun-Salamon conjecture, which asserts that all compact quaternion-K\"ahler manifolds are symmetric spaces. Our exposition is organised as follows.

In Section \ref{preliminariesonsp2sp1section} we review the basic properties of $\Sp(2)\Sp(1)$-structures, the geometry of which is determined by an algebraically special $4$-form $\Om$. By comparing with $\Spin(7)$-structures, determined by a different $4$-form $\Phi$, we derive several new identities. We emphasise the similarities and differences between the underlying structures.
In Section \ref{sec: harmonic qK structures} we derive the notion of harmonic $\Sp(2)\Sp(1)$-structures from the general framework of harmonic $H$-structures introduced in \cite{Loubeau2019}. We illustrate its relevance in \S\ref{examplesofharmonicQKsection},  by constructing explicit examples of harmonic $\Sp(2)\Sp(1)$-structures which are not torsion-free.

The analytic core of the paper is covered in Sections \ref{sec: harmonicQKflow} and \ref{sec: long-time and sing}. We formulate the quaternion-K\"ahler harmonic flow and the corresponding notion of soliton, studying basic properties such as evolution of torsion and parabolic rescalings. By relying upon the analogy with the harmonic $\Spin(7)$-flow  \cites{Dwivedi2021}, we establish in the quaternion-K\"ahler setting results such as a compactness theorem, almost-monotonicity formulae and $\epsilon$-regularity, as well as a description of the singular set of the flow.
One key difference of our approach is that we use the representation theory of $\Sp(2)\Sp(1)$, rather than overly involved local computations, to simplify several proofs and thus illustrate the usefulness of a more unified approach to harmonic $H$-flows. Moreover, by adapting the work of He-Li in \cite{HeLi2021} in the context of harmonic $\rU(n)$-structures, we derive an improved monotonicity formula, in tandem with a similar development in \cite{Fadel2022}; in fact, by further mobilising that paper's abstract theory for harmonic flows, we conclude long-time existence given small initial energy.

Finally, Section \ref{sec: explicit sols of the HF} illustrates different regimes of the harmonic flow with concrete examples. In \S\ref{sec: explicitflowsolutionsection} we study the flow on certain Lie groups and exhibit explicit solutions converging to harmonic $\Sp(2)\Sp(1)$-structures in infinite time. In particular, we exhibit a harmonic $\Sp(2)\Sp(1)$-structure on the manifold $\SU(3)$. In \S\ref{sec: steadysolitonexample} we construct an example of a steady gradient soliton of the flow, which to our knowledge is the first explicit non-trivial soliton of a harmonic flow of geometric structures.

\bigskip

\noindent\textbf{Acknowledgements:} 

The authors are grateful to Daniel Fadel, Eric Loubeau and Andr\'es Moreno, for a collaborative approach to this project, in tandem with the development of the general theory of geometric flows.

HSE was funded by the S\~{a}o Paulo Research Foundation (Fapesp)  \mbox{[2021/04065-6]} and the Brazilian National Council for Scientific and Technological Development (CNPq)  \mbox{[307217/2017-5]}. HSE has also benefited from a CAPES-COFECUB bilateral collaboration (2018-2022), granted by the Brazilian Coordination for the Improvement of Higher Education Personnel (CAPES) – Finance Code 001 [88881.143017/2017-01], and COFECUB [MA 898/18], and from a CAPES-MathAmSud (2021-2023) grant [88881.520221/2020-01]. 

UF was funded by the S\~{a}o Paulo Research Foundation (Fapesp) [2021/07249-0], in a postdoctoral grant subordinate to the thematic project \emph{Gauge theory and algebraic geometry} [2018/21391-1], led by Marcos Jardim.

\section{Preliminaries on \texorpdfstring{$\Sp(2)\Sp(1)$}{}-structures}
\label{preliminariesonsp2sp1section}

An $\Sp(2)\Sp(1)$-structure on an $8$-manifold $M$ is determined by an algebraically special $4$-form $\Om$, pointwise modelled on
\begin{gather}
	\Om=\frac{1}{2}(\om_1 \w \om_1+\om_2 \w \om_2+\om_3 \w \om_3),\label{qk4form}
\end{gather}
where the $2$-forms $\omega_1,\omega_2,\omega_3$ are given in local coordinates $\{x_i\}$ by
\begin{align}
	\om_1 &= dx_{12}+dx_{34}+dx_{56}+dx_{78},\label{symplecticform1}\\
	\om_2 &=dx_{13}-dx_{24}+dx_{57}-dx_{68},\label{symplecticform2}\\
	\om_3 &=dx_{14}+dx_{23}+dx_{58}+dx_{67}.\label{symplecticform3}
\end{align}
It might be worth recalling that $\Sp(2)$ is the stabiliser of the triple $\om_1,\om_2,\om_3$, and that the additional $\Sp(1)$ factor corresponds to rotating this data. Quotienting by the centre $\mathbb{Z}_2$, generated by $(-1,-1)$, leads to the $\Sp(2)\Sp(1):=\Sp(2) \times_{\mathbb{Z}_2} \Sp(1)$ structure. The study of harmonic $\Sp(2)$-structures turns out to be more subtle and is currently an ongoing project by the authors.

Since the stabiliser of $\Om$ in $\mathrm{GL}(8,\R)$ is isomorphic to $\Sp(2)\Sp(1)\subset\SO(8)$ (see also Proposition \ref{prop: metricfromOmproposition}), 
it follows that
$\Om$ defines, up to homothety, both a metric $g_\Om$ and volume form $\vol_{\Om}$. 
In the above notation, these are pointwise given by
\begin{gather}
	g_\Om=dx_1^2+dx_2^2+dx_3^2+dx_4^2+dx_5^2+dx_6^2+dx_7^2+dx_8^2,
	\label{qkmetric}\\
	\vol_{\Om}= \frac{1}{30} \Om \w \Om = dx_{1\dots8}.
	\label{qkvolumeform}
\end{gather}
The action of $\Sp(2)\Sp(1)$ on $\R^8$ corresponds to the usual left-action of $\Sp(2)$ on $\mathbb{H}^2$ and right-action by $\Sp(1)$. This representation can be seen more concretely be way of Salamon's $E$-$H$ formalism, as follows.
The complexified (co)tangent bundle can be viewed as the $\Sp(2)\Sp(1)$-module
\begin{equation}
\label{eq: tangentbundle}
    T^*_\C M= E \otimes H,
\end{equation}
where $E$ (respectively $H$) is the associated vector bundle to the standard representation of $\Sp(n)$ (respectively $\Sp(1)$) on $\mathbb{C}^{2n}$ (respectively $\mathbb{C}^2$), see also \cite{Salamon1982}. 
In what follows we shall often ignore the fact that we are complexifying the tensor bundles of $M$, and use the same notation for the complexified and the underlying real vector bundles, as all these spaces admit a real structure owing to the quaternionic structure.

\subsection{Representation theory and intrinsic torsion}

A description of the tensor bundles on quaternion-K\"ahler $8$-manifolds in the $E$-$H$ notation can be found in \cites{Salamon1989, Swann1991}, but for our purposes we shall need the following more concrete description, found in \cite{FowdarSalamon2021}. 
The space of $2$-forms splits as an $\Sp(2)\Sp(1)$-module as follows: 
\begin{align*}
\Lm^2 &= \mathfrak{sp}(1)\oplus \mathfrak{sp}(2)\oplus (\mathfrak{sp}(1)\oplus \mathfrak{sp}(2))^\perp \\
&= S^2H \oplus S^2E \oplus \Lm^2_0 E \otimes S^2 H\\
&= \Lm^2_3 \oplus \Lm^2_{10} \oplus \Lm^2_{15}
\end{align*} 
where
\begin{align}
	\Lm^2_3  &= \{\al\in \Lm^2 \ | * (\al \w \Om)=5\al  \},\label{Lm23} \\
	\Lm^2_{10}  &= \{\al\in \Lm^2 \ | * (\al \w \Om)=-3\al  \},\label{Lm210} \\
	\Lm^2_{15} &= \{ \al\in \Lm^2 \ | * (\al \w \Om)=\al  \}.\label{Lm215}
\end{align}
Note that the subbundle $\Lm^2_3$ can be equivalently defined as the span $\langle\om_1,\om_2,\om_3\rangle$, although we should emphasise that the $2$-forms $\om_i$ can only be chosen locally, i.e. $\Lm^2_3$ is not in general a trivial vector bundle (for instance consider $M=\mathbb{H}\mathbb{P}^2$). 

We shall also need the decomposition of the space of $4$-forms. First note that the Hodge star operator $*$ splits the space of $4$-forms into self-dual and an anti-self-dual components:
$$\Lm^4=\Lm^{4+} \oplus \Lm^{4-}.$$
These further decompose into irreducible $\Sp(2)\Sp(1)$-modules as follows:
\begin{align*}
	\Lm^{4+}_1 &= \langle  \Om \rangle \\ 
	\Lm^{4+}_5 \cong S^4H 
	&= \{\al\in \Lm^{4+} \ |  \ * (\al \w \om) \w \Om =5 \al \w \om \ \ \forall \om \in \Lm^2_3  \}\\
	\Lm^{4+}_{15} \cong \Lm^2_0 E \otimes S^2 H
	&= \{\al\in \Lm^{4+} \ | \ * (\al \w \om) \w \Om = \al \w \om \ \ \forall \om \in \Lm^2_3  \}\\
	\Lm^{4+}_{14} \cong S^2_0(\Lm^2_0 E) &= \{\al \in \Lm^{4+}\ |\  \al \w \om=0 \ \ \forall \om \in \Lm^2_3  \} \\
	\Lm^{4-}_{5}\cong \Lm^2_0 E 
	&= \{\al\in \Lm^{4-} \ | \ * (\al \w \om) \w \Om = \al \w \om \ \ \forall \om \in \Lm^2_3  \}\\
	\Lm^{4-}_{30} \cong S^2 E \otimes S^2 H 
	&= \{\al\in \Lm^{4-} \ | \ * (\al \w \om) \w \Om =-3 \al \w \om \ \ \forall \om \in \Lm^2_3  \}.
\end{align*}
Recall also that there is a natural action of $\Lm^2 \cong \mathfrak{so}(8) \subset \mathrm{End}(\R^8)$ on the space of $k$-forms given by
\begin{align*}
	\Lm^2 \otimes \Lm^k &\to \Lm^k\\
	(\al \w \beta)\otimes \Upsilon &\mapsto \al \w (\beta^\sharp \ip \Upsilon)-\beta \w (\alpha^\sharp \ip \Upsilon).
\end{align*}
When $k=2$, observe that this is just the usual Lie bracket operation in $\mathfrak{so}(8)$. Since in our situation we have the quaternion-K\"ahler $4$-form $\Om$, this gives rise to an `infinitesimal action' operator $\diamond: \Lm^2 \to \Lm^4$, defined on simple $2$-forms by
\begin{equation}
    (\alpha \w \beta)  \mapsto (\alpha \w \beta) \diamond \Om 
    := \al \w (\beta^\sharp \ip \Om)-\beta \w (\alpha^\sharp \ip \Om)\label{diamondoperator} .\end{equation}
\begin{lemma}
\label{kernelofdiamond}
	The kernel of the operator $\diamond$ is isomorphic to $S^2H\oplus S^2 E$, hence $\diamond$ restricts to an isomorphism on $\Lm^2_{15}\cong \Lm^{4+}_{15}$.
\end{lemma}
\begin{proof}
	Since $\diamond$ is an $\Sp(2)\Sp(1)$-equivariant map, and $\Lm^2_0 E \otimes S^2H$ is the only $\Sp(2)\Sp(1)$-module contained in both $\Lm^2$ and $\Lm^4$, by Schur's lemma it suffices to check that it is non-zero.
\end{proof}
\begin{remark}\label{defintionofdiamond}
More generally, the diamond operator $\diamond$ can be naturally extended to the action of $$\mathfrak{gl}(8,\R)\cong \mathrm{End}(\R^8) \cong \R^{8*}\otimes \R^8 \cong S^2(\R^{8*})\oplus \Lm^2(\R^{8*})$$ 
on $\Om$ as above by $(\alpha \otimes \beta) \diamond \Om := \al \w (\beta^\sharp \ip \Om)$. Then the same argument as in Lemma \ref{kernelofdiamond} shows that there is an isomorphism 
$$
S^2(T^*M) \cong \langle \Om\rangle \oplus \Lm^{4-}_5\oplus \Lm^{4-}_{30}.
$$ 
\end{remark}
The intrinsic torsion tensor $T \in \Lm^1_8 \otimes \Lm^2_{15}$ of the $\Sp(2)\Sp(1)$-structure determined by $\Om$ can now be defined as
\begin{equation}
\nabla _{\cdot}\Om= T(\cdot) \diamond \Om \in \Lm^1_8 \otimes \Lm^{4+}_{15}.
\end{equation}
In order to extract $T$ from the above expression we need to invert the isomorphism $\diamond: \Lm^2_{15} \to \Lm^{4+}_{15}$. To do so we first note that given an arbitrary $4$-form $\kappa$ one can define a triple contraction operator $\ip_{3}:\Lm^4 \to \Lm^2$, given on simple $4$-forms, by
\begin{align}
	\alpha_1 \w \alpha_2 \w \alpha_3 \w \alpha_4\ \ip_{3} \kappa 
	:= \
	&\al_1 \w (\alpha_2^\sharp \ip \alpha_3^\sharp\ip \alpha_4^\sharp\ip\ \kappa)-
	\al_2 \w (\alpha_1^\sharp \ip \alpha_3^\sharp\ip \alpha_4^\sharp\ip \kappa)\ +\\
	&\al_3 \w (\alpha_1^\sharp \ip \alpha_2^\sharp\ip \alpha_4^\sharp\ip \kappa)-
	\al_4 \w (\alpha_1^\sharp \ip \alpha_2^\sharp\ip \alpha_3^\sharp\ip \kappa).\nonumber
\end{align}
In particular, by taking $\kappa=\Om$ we get the following $\Sp(2)\Sp(1)$-equivariant map
\begin{align}
	\iota_3(\alpha_1 \w \alpha_2 \w \alpha_3 \w \alpha_4) :=
	 \alpha_1 \w \alpha_2 \w \alpha_3 \w \alpha_4\ \ip_{3} \Om. 
	\nonumber
\end{align}
By inspecting the decomposition of  $2$- and $4$-forms into irreducible $\Sp(2)\Sp(1)$-modules, and by Schur's lemma, we know that either $\iota_3$ is zero or it restricts to an isomorphism $\Lm^{4+}_{15} \to \Lm^2_{15}$. 
\begin{lemma}
\label{triplecontractioninverse}
	The operator $\iota_3$ satisfies
	\begin{equation}
		\iota_3(\kappa \diamond \Om) =32 \kappa,
		\qforq
		\kappa \in \Lm^2_{15}.
	\end{equation}
\end{lemma}
\begin{proof}
	Since the result is algebraic, it suffices to work at a point in $M^8$.
	Given a simple $2$-form $\alpha \w \beta$, we want to compute $(\alpha \w \beta) \diamond \Om$. We now make two convenient assumptions, without loss of generality. Since $\Om$ is invariant by $\Sp(2)\Sp(1)$, and $\Sp(2)$ acts transitively on $S^7$,  we can set $\al=dx_1$ while leaving $\Om$ unchanged; furthermore, as the stabiliser of $dx_{1}$ in $\Sp(2)$ is isomorphic to $\Sp(1)$, we can assume that $\beta= b\cdot dx_{2}+c\cdot dx_{3}+u\cdot dx_{4}+v\cdot dx_{5}$, for some constants $b,c,u,v$. Hence our typical $2$-form can be written as
	 \[\al \w \beta = \frac{b}{4}\cdot \om_1 +\frac{c}{4}\cdot \om_2 +\frac{u}{4}\cdot \om_3 +\pi^2_{10}(\al \w \beta)+\pi^2_{15}(\al \w \beta),\]
	 where $\pi^i_{j}:\Lm^i \to \Lm^i_{j}$ denotes the projection map.
	A direct computation now shows that
	\[ \iota_3((\alpha \w \beta) \diamond \Om) = 32\cdot \pi^2_{15}(\alpha \w \beta). 
	\qedhere\]
\end{proof}

Using Lemma \ref{triplecontractioninverse} we can now rewrite the intrinsic torsion tensor as
\begin{equation}
\label{eq: qktorsion}
    T(\cdot)= \frac{1}{32}\  \iota_3(\nabla_{\cdot} \Om).
\end{equation}
Finally we also record one key identity between the diamond operator and the  triple contraction, for later use.
\begin{lemma}
\label{Lm215andLm215pairingLemma}
    Given $2$-forms $\alpha=\sum_{i<j}\alpha_{ij}dx_{ij}$
    and $\beta=\sum_{i<j}\beta_{ij}dx_{ij}$ belonging to $\Lm^2_{15}\cong \Lm^2_0E\otimes S^2H$, we have 
\begin{equation}
    (\alpha \diamond \Om) \ip_3 (\beta \diamond \Om)= 32 \sum_{i,j,k} \alpha_{ik}\beta_{jk}dx_{ij} \in \Lm^2_3 \oplus \Lm^2_{10}.\label{Lm215andLm215pairing}
\end{equation}
\end{lemma}
\begin{proof}
Again this is an algebraic relation, which can be assessed at a point. By choosing geodesic normal coordinates, we can assume that we are working on $(\R^8,\Om)$.
Since $\diamond$ and $\ip_3$ are both linear operators, it suffices to consider the case when $\alpha$ and $\beta$ are of the form $v\otimes w \in \Lm^2_{15}\cong \Lm^2_0E\otimes S^2H$, for some unit vectors $v,w$. Furthermore we know that $\Sp(2)\Sp(1)$ acts as $\SO(5)\cong \Sp(2)/\mathbb{Z}_2$ on $\Lm^2_0E \cong \R^5$ and as $\SO(3)\cong \Sp(1)/\mathbb{Z}_2$ on $S^2H\cong \R^3$ i.e. $\Sp(2)\Sp(1)$ acts transitively on each unit sphere and hence, as in the proof of Lemma \ref{triplecontractioninverse},  we can choose any elements $\al$ and $\beta$ in $\Lm^2_{15}$. So for instance we consider $\alpha=dx_{28}+dx_{35}$ and $\beta=dx_{15}-dx_{26}$; using characterisation (\ref{Lm215}) for $\Lm^2_{15}$, one readily checks that:
\begin{align*}
    \Lm^2_{15}=\langle &dx_{15}-dx_{26},dx_{16}+dx_{25},dx_{15}-dx_{37},dx_{16}+dx_{38},dx_{28}+dx_{35},dx_{17}-dx_{28},\\ 
    &dx_{18}+dx_{45}, dx_{27}+dx_{36},dx_{28}+dx_{46},dx_{38}+dx_{47},dx_{37}-dx_{48},dx_{18}+dx_{27},\\ &dx_{12}+dx_{34}-dx_{56}-dx_{78},dx_{13}-dx_{24}-dx_{57}+dx_{68},dx_{14}+dx_{23}-dx_{58}-dx_{67}\rangle.
\end{align*}
    By direct computation, we find
\begin{align*}
    \alpha \diamond \Om 
    &= 4 (dx_{1245}-dx_{1348}-dx_{2567}+dx_{3678}) \\
    \beta \diamond \Om 
    &= 4 (dx_{1346}+dx_{1678}+dx_{2345}+dx_{2578}),
\end{align*}
    from which we deduce that
\[
(\alpha \diamond \Om) \ip_3 (\beta \diamond \Om)=32(-dx_{13}+dx_{68}),
\]
as required (compare with Remark \ref{symmetricspaceRem}, below). 

    Now, as $\Sp(2)\Sp(1)$-modules, we have the following decomposition: 
\[\Lm^2_{15}\otimes \Lm^2_{15}\cong (\R \oplus S^2E \oplus S^2_0(\Lm^2_0E)) \otimes (\R \oplus S^2H \oplus S^4 H).
\]
Observe that there is no $\Lm^2_{15}\cong\Lm^2_0 E \otimes S^2H$ component, but there are $\Lm^2_3 \cong S^2H$ and  $\Lm^2_{10}\cong S^2E$ components, so it follows that $(\alpha \diamond \Om) \ip_3 (\beta \diamond \Om) \in \Lm^2_3 \oplus \Lm^2_{10}.$ Indeed, in the above example, 
\[2(dx_{13}-dx_{68})=\om_{2}+(dx_{13}+dx_{24}-dx_{57}-dx_{68})\in \Lm^2_3 \oplus \Lm^2_{10}.\]
The fact that the last term lies in $\Lm^2_{10}$ is easily checked using (\ref{Lm210}), and this concludes the proof.
\end{proof}
\begin{remark}
\label{symmetricspaceRem}
    The right-hand side of (\ref{Lm215andLm215pairing}) corresponds, up to a constant factor, to the Lie bracket of $\alpha, \beta \in \Lm^2_{15} \subset \Lm^2 \cong \mathfrak{so}(8)$. Indeed it is well-known that the splitting $\mathfrak{so(8)}\cong(\mathfrak{so}(3)\oplus\mathfrak{so}(5))\oplus \Lm^2_{15}$ corresponds to the Lie algebra decomposition for the rank $3$ symmetric space $\frac{\SO(8)}{\SO(3)\times \SO(5)}$ (which is the double cover of $\frac{\SO(8)}{\Sp(2)\Sp(1)}$) and hence
\[[\Lm^2_{15},\Lm^2_{15}]\subset\mathfrak{so}(3)\oplus\mathfrak{so}(5)\]
    cf. \cite{Kobayashi1969a}*{Ch. XI. Prop 2.1}. Thus, Lemma \ref{Lm215andLm215pairingLemma} expresses essentially just a consequence of this fact.
\end{remark}

While in this paper we aim to deal with situations in which $T$ does not vanish identically, it is worth recalling some properties of torsion-free quaternion-K\"ahler structures (in all dimensions). Quaternion-K\"ahler manifolds are always Einstein i.e. $\Ric(g)= \lambda g$ cf. \cites{Besse2008, Salamon1982}. If moreover $\lambda=0$, then $M$ is locally a hyperK\"ahler manifold, so this case is usually excluded from the definition of quaternion-K\"ahler manifolds. If $\lambda >0$, then $M$ is compact, while if $\lambda <0$ then $M$ is non-compact. Poon and Salamon showed that the only compact quaternion-K\"ahler $8$-manifolds are the symmetric spaces \cite{PoonSalamon1991}:
$$
\mathbb{H}\mathbb{P}^2=\frac{\Sp(3)}{\Sp(2)\Sp(1)},\ \ \mathrm{Gr}_2(\C^4)=\frac{\SU(4)}{\mathrm{S}(\U(2)\U(2))}\qandq \frac{\mathrm{G}_2}{\SO(4)}.
$$ 
By contrast, LeBrun showed in \cite{LeBrunInfiniteQK} that there are infinitely many examples in the non-compact case, see also \cites{Besse2008, UdhavFowdar4} for other non-compact examples.

Furthermore, by analysing the decomposition
\begin{equation}
\label{eq: torsion decomposition}
    T \in \Lm^1_8 \otimes \Lm^2_{15} \cong \Lm^1_8 \oplus \Lm^3_{16}\oplus \Lm^3_{32} \oplus (K \otimes S^3 H),
\end{equation}
where $K$ is irreducible $\Sp(2)\Sp(1)$-module defined by $\Lm^2_0 E \otimes E \cong K \oplus E$, Swann proved that 
\begin{theorem}[\cite{Swann1991}]
	The intrinsic torsion $T=0$ if, and only if, $d\Om=0$ and the differential ideal  $\langle\om_1,\om_2,\om_3\rangle$ is algebraic.
\end{theorem}

While, for quaternion-K\"ahler structures in dimensions strictly greater than $8$, being torsion-free is equivalent to the $4$-form $\Om$ being closed, there do exist quaternion-K\"ahler $4$-forms in dimension $8$ which are closed but not torsion-free \cite{Salamon2001}. 
\begin{remark}
Since, in dimension $8$, $\Om$ is a self-dual $4$-form, if $d\Om=0$ then one often calls the induced quaternion-K\"ahler structure \emph{harmonic}  \cite{ContiMadsen2015}, which is is an altogether different meaning from our notion of harmonicity in the present context.
\end{remark}

$\Spin(7)$-structures are another type of geometric structure arising on $8$-manifolds by an algebraically special $4$-form, under favourable topological conditions. Throughout this article we shall see that there is a rather close relation between harmonic $\Sp(2)\Sp(1)$- and $\Spin(7)$-structures, although their respective algebraic properties are quite different. Next we describe some common features of $\Sp(2)\Sp(1)$- and $\Spin(7)$-structures which, to the best of our knowledge, have not so far been described in the literature.

\subsection{Bianchi identity for the torsion}
We now derive a `Bianchi-type identity' for the torsion tensor $T$, which will be useful later on in the derivation of a monotonicity formula. The terminology comes the fact that this identity arises due to the diffeomorphism-invariance of the torsion tensor, just as the usual Bianchi identity arises from the invariance of Riemann curvature, cf. \cite{Karigiannis2007}. 

\begin{proposition}
\label{bianchiidentityproposition}
    The torsion tensor $T$ satisfies the following `Bianchi-type identity'
\begin{equation}
\label{eq: bianchi identity}
    (\nabla_{X}T)(Y)-(\nabla_{Y}T)(X) = \pi^{2}_{15}(R(X,Y))+ \frac{1}{32}\big((\nabla_{Y} \Om)\ip_{3}(\nabla_{X}\Om)-(\nabla_{X} \Om)\ip_{3}(\nabla_{Y}\Om)\big),
\end{equation}
    where we are viewing $R(Y,X)$ as a $2$-form. Moreover,
\begin{equation}
\label{eq: bianchi identity2}
    \pi^{2}_{15}((\nabla_{X}T)(Y)-(\nabla_{Y}T)(X)) = \pi^{2}_{15}(R(X,Y)).
\end{equation}
\end{proposition}
\begin{proof}
    From (\ref{eq: qktorsion}), we have
    \begin{equation}
    (\nabla_{Y}T)(X)=\frac{1}{32}(\nabla^{2}_{Y,X}\Om) \ip_{3} \Om + \frac{1}{32}(\nabla_{X} \Om)\ip_{3}(\nabla_Y\Om),\end{equation}
    where we used the fact that $\nabla$ preserves $g$ and hence $\ip_{3}$. Skewsymmetrising in $X$ and $Y$, we get
\begin{equation}
    (\nabla_{X}T)(Y)-(\nabla_{Y}T)(X)=\frac{1}{32}((\nabla^{2}_{X,Y}-\nabla^{2}_{Y,X}) \Om) \ip_{3} \Om + \frac{1}{32}\big((\nabla_{Y} \Om)\ip_{3}(\nabla_{X}\Om)-(\nabla_{X} \Om)\ip_{3}(\nabla_{Y}\Om)\big).
\end{equation}
    and the first part of the Proposition now follows from Lemma \ref{triplecontractioninverse}. 
    
    For the second part,  observe that
\begin{equation}
    (\nabla_{Y} \Om)\ip_{3}(\nabla_{X}\Om) \in \Lm^{4+}_{15} \otimes \Lm^{4+}_{15} \cong (\R \oplus S^2_0(\Lm^2_0 E) \oplus S^2 E)\otimes (\R \oplus S^2 H \oplus S^4 H).\label{decompositionofsquareofOm215}
\end{equation}
In particular,  $(\nabla_{Y} \Om)\ip_{3}(\nabla_{X}\Om)$ has no component in $\Lm^{2}_{15} \cong \Lm^2_0 E \otimes S^2H$, and hence as a $2$-form it lies entirely in $S^2 H \oplus S^2 E\subset \Lm^2$; likewise for $(\nabla_{X} \Om)\ip_{3}(\nabla_{Y}\Om)$. This concludes the proof.
\end{proof}
An important consequence of the Bianchi identity (\ref{eq: bianchi identity}) is that the skew-symmetrisation of the covariant derivative of $T$ is fully controlled by the $15$-dimensional component of the curvature tensor (which depends only on the metric) and a quadratic term involving $T$.

\begin{remark}
    In \cite{Dwivedi2021} an analogous Bianchi-type identity is derived for the torsion of a $\Spin(7)$-structure, say determined by $\Phi$. The proof there is more computational in nature but the argument follows exactly as described above by replacing $\Om$ by $\Phi$ and $\Om^{2}_{15} \cong (\mathfrak{sp}(2)\oplus \mathfrak{sp}(1))^{\perp}$ by $\Om^2_7 \cong \mathfrak{spin(7)}^{\perp}$; so by analogy with (\ref{decompositionofsquareofOm215}), we have the $\Spin(7)$-module decomposition
\[\Lm^2_7 \otimes \Lm^2_7 \cong \R \oplus S^2_0(\R^7)\oplus \mathfrak{so}(7),\]
    corresponding in fact to the  representation of $\SO(7)\cong \Spin(7)/\mathbb{Z}_2$.

    More generally, suppose that we have the orthogonal reductive splitting  $\mathfrak{so}(n)=\mathfrak{h}\oplus\mathfrak{m}$ and that $H \cong \text{stab}(\xi)$ for some tensor $\xi$ (in our case $G=\SO(8)$, $H=\Sp(2)\Sp(1)$ and $\xi = \Om$), then we know that the torsion tensor $T \in \Om^1 \otimes \mathfrak{m}$ cf. \cite{Salamon1989}. The last term in (\ref{eq: bianchi identity}) essentially corresponds to the Lie bracket of $T(X)$ and $T(Y)$ and hence belongs to $[\mathfrak{m},\mathfrak{m}]$. So, if $\mathfrak{g}=\mathfrak{h}\oplus\mathfrak{m}$ corresponds to the Lie algebra decomposition of a symmetric space, then $[\mathfrak{m},\mathfrak{m}]\subset \mathfrak{h}$. This is indeed the case in our situation and also eg. when $G=\SO(8)$ with $H=\Spin(7)$ \cite{Dwivedi2021}, $G=\SO(2n)$ and $H=\U(n)$ \cite{HeLi2021}. Thus, such a Bianchi identity must always hold in those contexts. This insight allows us to interpret proofs in these various contexts from a unified perspective, and thereby avoid unnecessarily complicated computations, as we shall illustrate below.
\end{remark}

\begin{corollary}
    If the intrinsic torsion $T$ of $(M^8,g_{\Om},\Om)$ vanishes, then the holonomy group of $g_{\Om}$ is contained in $\Sp(2)\Sp(1)$. Moreover, $g_{\Om}$ is Einstein. 
\end{corollary}
\begin{proof}
    Setting $T=0$ in (\ref{eq: bianchi identity}) shows that $\pi^2_{15}(R(X,Y))=0$, for all $X,Y \in T_pM$, i.e. the curvature tensor $R$ corresponds to a section of $S^2(\mathfrak{sp}(2)\oplus\mathfrak{sp}(1)) \subset S^2(\Lm^2)$. The first claim now follows from the Ambrose-Singer Theorem. 
    
    To establish the second claim, recall that the curvature operator in fact lies in the kernel of the skew-symmetrisation map $$A:S^2(\mathfrak{sp}(2)\oplus\mathfrak{sp}(1)) \to \Lm^4$$ defined by wedging the $2$-forms in $\mathfrak{sp}(2)\oplus\mathfrak{sp}(1)$ cf. \cite{Salamon1989}; this corresponds to the symmetry of the algebraic Bianchi identity. On the other hand, we have the irreducible decomposition
\[S^2(\mathfrak{sp}(2)\oplus \mathfrak{sp}(1))\cong S^4E \oplus S^2_0(\Lm^2_0 E) \oplus \Lm^2_0 E\oplus \R \oplus S^2 E \otimes S^2H \oplus S^4 H \oplus \R \]
    where we again use the $E$-$H$ formalism of \eqref{eq: tangentbundle}. The traceless component of the Ricci tensor belongs to $S^2_0(E \otimes H) \cong \Lm^2_0 E \oplus S^2 E \otimes S^2H$. Comparing with the irreducible decomposition of $\Lm^4$, we see that the kernel of $A$ always contains a copy of $\R$ (the  curvature tensor of $\mathbb{H}\mathbb{P}^2$) and of $S^4E$. Testing a few simple examples shows that the map $A$ has a non-zero image in each irreducible component of $\Lm^4$; for instance one can consider the wedge products of $\om_i \in \mathfrak{sp}(1)$ and $dx_{12}-dx_{34} \in \mathfrak{sp}(2)$. Hence from Schur's Lemma it follows that $A$ must be an isomorphism on all the remaining modules and this gives the result.
\end{proof}
\begin{remark}
    A similar argument was used in \cite{Karigiannis2007}*{Corollary 4.12} to give a direct proof that $\mathrm{G}_2$-manifolds are indeed Ricci-flat (although the proof therein relies on a calculation in index notation for the $\mathrm{G}_2$-structure  $3$-form $\vp$, the essence is the same). Our argument above shows that in fact, given a Bianchi-type identity, a similar proof can be used to show Ricci flatness for other special holonomy groups, cf. \cite{Fadel2022}.
\end{remark}

\subsection{Relations between $\Sp(2)\Sp(1)$-structures and \texorpdfstring{$\Spin(7)$}{}-structures}
A $\Spin(7)$-structure on an $8$-manifold $M$ is determined by an algebraically special $4$-form $\Phi$ pointwise modelled on
\begin{equation}
\label{eq: spin7form}
    \Phi = \frac{1}{2}(-\om_1 \w \om_1+\om_2 \w \om_2+\om_3 \w \om_3 ),
\end{equation}
with $\om_i$ as defined above by (\ref{symplecticform1})-(\ref{symplecticform3}). Since $\Spin(7)$ is a subgroup of $\SO(8)$, it follows that $\Phi$ determines (up to homothety) both a metric $g_{\Phi}$ and a volume form 
\[\vol_{\Phi}=\frac{1}{14}\Phi \w \Phi.\] 
In the above pointwise coordinates, these coincide with the expressions (\ref{qkmetric}) and (\ref{qkvolumeform}), respectively. It is worth pointing out that a $\Spin(7)$-structure endows each tangent space of $M$ with the algebraic structure of the octonions $\mathbb{O}$, while an $\Sp(2)\Sp(1)$-structure endows the tangent space with the algebraic structure of the quaternic plane $\mathbb{H}^2.$

As demonstrated by Karigiannis in \cite{Karigiannis2005}*{Theorem 4.3.5}, the metric $g_\Phi$ can be explicitly extracted from $\Phi$ via the expression
\begin{equation}
\label{eq: spin7metricfromPhi}
    g_\Phi(X,X)^2
    =\frac{7^3}{6^{7/3}}
    \frac{(\det (((e_i \ip X \ip \Phi) \w (e_j \ip X \ip \Phi) \w (X \ip \Phi))(e_1,\dots,e_7)))^{1/3}}{(((X \ip \Phi)\w \Phi))(e_1,\dots,e_7))^3 },
\end{equation}
where $X,e_i \in T_p N$ form a positively oriented basis of $T_pN$, i.e. $$\vol_\Phi(X,e_1,\dots,e_7)>0.$$
In fact, \cite{Karigiannis2005}*{Lemma 4.3.3} shows that the right-hand side of (\ref{eq: spin7metricfromPhi}) is independent of the choice of extension of $X$ to the basis $\{X,e_i\}$
of $T_p N$: if one chooses a different extension $\{e_i'\}$ so that 
\[e'_i=P_{ij}e_j+Q_iX,\]
then the numerator of (\ref{eq: spin7metricfromPhi}) changes by a factor of $(\det(P)^2\det(P)^7)^{1/3}$, but so does also the denominator and hence the quotient is indeed invariant.

An inspection of the proof of the latter assertion reveals that the invariance of the right-hand side still holds if $\Phi$ is replaced by any $4$-form $\Upsilon$ which is non-degenerate, i.e. $\Upsilon \w \Upsilon >0$. In particular, this  leads to the following analogous result:
\begin{proposition}
\label{prop: metricfromOmproposition}
    The quaternion-K\"ahler metric $g_\Om$ is obtained from $\Om$ via the expression
\begin{equation}
    g_\Om(X,X)^2
    =\frac{5^3}{4 \cdot 6^{1/3}}
    \frac{(\det (((e_i \ip X \ip \Om) \w (e_j \ip X \ip \Om) \w (X \ip \Om))(e_1,\dots,e_7)))^{1/3}}{(((X \ip \Om)\w \Om))(e_1,\dots,e_7))^3 },\label{eq: qkmetricfromOm}
\end{equation}
    where $X \in T_pN$ and $\{X, e_i\}\in T_pN$ is any extension of $X$ to a postively oriented basis of $T_pN$.
\end{proposition}
\begin{proof}
    We know from \cite{Karigiannis2005}*{Lemma 4.3.3} that the right-hand side of \eqref{eq: qkmetricfromOm} is independent of the extension $\{e_i\}$, and it is $\Sp(2)\Sp(1)$-invariant, so it suffices to check that \eqref{eq: qkmetricfromOm} holds for a preferred extension. Identifying $\Om$ at $p$ with the standard $\Om_0$ on $\R^8$, we let $X$ be an arbitrary vector, $e_i:=I_i(X)$, for $i=1,2,3$, and we choose $\{e_4,e_5,e_6,e_7\}$ to be orthogonal to $\{X,e_1,e_2,e_3\}$. The result will now follow from the next Lemma.
\end{proof}
\begin{lemma}
    Given $V,W\in \Gamma(TN)$, consider the orthogonal decomposition $W=a V + W_1 + W_2$, 
    where $W_1$ denotes the projection of $W$ onto $\langle I_1V, I_2V, I_3V\rangle$ and $W_2$ denotes the projection of $W$ onto the orthogonal complement of the quaternionic span of $V$. Then we have
\begin{equation}
\label{qkmetricon2plane}
    (V \ip W \ip \Om)\w (V \ip W \ip \Om) \w \Om 
    = 6\cdot g_\Om(V,V)\cdot (3\cdot g_\Om(W_1,W_1) - g_\Om(W_2,W_2))\cdot  \vol_\Om.
\end{equation}
In particular, given another vector field $U=bV+U_1+U_2$, by polarising the above we get
\begin{equation}
    (V \ip W \ip \Om)\w (V \ip U \ip \Om) \w \Om = 6\cdot g_\Om(V,V)\cdot (3\cdot g_\Om(W_1,U_1) - g_\Om(W_2,U_2))\cdot  \vol_\Om.
    \end{equation}
\end{lemma}
\begin{proof}
    Since the result is algebraic, we may work on $\R^8$ without loss of generality. Furthermore, $\Sp(2)$ acts transitively on the unit sphere, so we can assume $V=\partial_{x_1}$, and hence $I_1(V)=\partial_{x_2}$, $I_2(V)=\partial_{x_3}$ and $I_3(V)=\partial_{x_4}$. Since the stabiliser in $\Sp(2)$ of a unit vector is isomorphic to $\Sp(1)$, we can also set $W_2=\partial_{x_5}$.
The result now follows from a straightforward computation.
\end{proof}
\begin{remark}
    Given a unit vector $V \in T_p N$, the subspace $\langle I_1V, I_2V, I_3V \rangle$ can now be defined as the span of those unit vectors $W$ for which the right-hand side of (\ref{qkmetricon2plane}) is equal to $18 \vol_\Om$. This gives a concrete way of defining the $2$-sphere of almost complex structures, i.e. the twistor space, starting from $\Om$ only.
\end{remark}

\section{Harmonic  \texorpdfstring{$\Sp(2)\Sp(1)$}{}-structures}
\label{sec: harmonic qK structures}

In this section we describe how the notion of harmonic $\Sp(2)\Sp(1)$-structures arises from the general framework of harmonic $H$-structures introduced in \cite{Loubeau2019}. 

\subsection{Harmonic homogeneous  $\Sp(2)\Sp(1)$-sections}
\label{sec: harmonicsp2sp1sections}

We begin by describing $\Sp(2)\Sp(1)$-structures as sections of a homogeneous fibre bundle. We shall be brief here and refer the reader to \cite{Loubeau2019} for more details.
First we fix an oriented Riemannian $8$-manifold $(M,g)$ and denote by $p:F\to M$ its orthonormal frame bundle, with fibre $G=\SO(8)$. Since $\Sp(2)\Sp(1)\subset \SO(8)$, cf. Proposition \ref{prop: metricfromOmproposition},
the quotient $q: F\to N:=F/\Sp(2)\Sp(1)$ defines a principal $H=\Sp(2)\Sp(1)$-bundle, which in turn is a smooth fibre bundle $\pi : N \to M$ with  the homogeneous space $\SO(8)/\Sp(2)\Sp(1)$ as typical fibre. It follows that $\Sp(2)\Sp(1)$-structures compatible with the metric $g$ are determined by sections of $\pi$.

From the results in Section \ref{preliminariesonsp2sp1section}, we have the reductive splitting
$$
\mathfrak{so}(8)=\mathfrak{sp}(2)\oplus\mathfrak{sp}(1)\oplus\fm 
\qandq 
\mathrm{Ad}\Big|_{\Sp(2)\Sp(1)}\fm\subseteq\fm,
$$
where $\mathfrak{m}\cong \Lm^2_{15}(\R^8)$ as described above. In what follows we shall simply write $\Lm^2_{15}$ for $\Lm^2_{15}(\R^8)$. 
The Levi-Civita connection $\omega\in\Omega^1(F,\mathfrak{so}(8))$ on the frame bundle induces the splitting into vertical and horizontal components
$$
TN=\cV \oplus \cH
$$
with
$
\cV:=\ker\pi_*=q_*(\ker p_*) 
$
and
$
\cH:=q_*(\ker\omega).
$

\begin{multicols}{2}
\label{figure 1}
Let $\underline{\Lm^2_{15}}\to N$ be the vector bundle associated to $q$ with fibre $\Lm^2_{15}$, whose points are the $\Sp(2)\Sp(1)$-equivalence classes defined by the infinitesimal action of $w\in {\Lm^2_{15}}$ on $z\in F$ i.e.
$$
z\bullet w:=[(z,w)]_{H}=\cI(q_*(w_z^*)) \in \underline{\Lm^2_{15}} := F\times_H {\Lm^2_{15}},
$$
where $w_z^*$ is a fundamental left-invariant vector field.
\columnbreak

$$\xymatrix{
\llap{$z\in$ } F 
	\ar[dr]^q 
    \ar[dd]_p
    & & {\underline{\Lm^2_{15}}} \ar[dl] \\ 
  &\llap{$y\in$ } N 
  \ar[dl]_\pi 
  	& \ar[l]\cV
  	\ar[u] ^\cI  \\ 
\llap{$x\in$ }  M 
& & }$$
\end{multicols}
This defines a vector bundle isomorphism
\begin{equation}
\label{eq: isomorphism I:V->m}
\begin{array}{rcccl}
    \cI&:&\cV&\tilde{\rightarrow}&\underline{\Lm^2_{15}}\\
    &&q_*(w_z^*)&\mapsto&z\bullet w 
\end{array}
\end{equation}
and, since the ${\Lm^2_{15}}$-component $\omega_\fm \in \Omega^1(F,{\Lm^2_{15}})$ of the Levi-Civita connection is $\Sp(2)\Sp(1)$-equivariant and $q$-horizontal, it projects to a \emph{homogeneous connection form} $f \in\Omega^1(N,\underline{\Lm^2_{15}})$ defined by:
\begin{equation}
\label{eq: homog connection form f}
    f (q_*(Z)):= z \bullet \omega_\fm(Z) 
    \qforq
    Z\in T_zF.
\end{equation}

\newpage
\begin{multicols}{2}
Recall that a vector $v\in TN$ describes an incidence condition at a $\Sp(2)\Sp(1)$-class of frames. On $\pi$-vertical vectors, the connection form $f $ coincides with the canonical isomorphism (\ref{eq: isomorphism I:V->m}), while $\pi$-horizontal vectors lie in the kernel i.e. 
$$
f (v_y)=\cI(v_y^\cV), 
\qforq
v_y\in T_yN.
$$
\noindent NB.: in the adjacent diagram, $\fh = \mathfrak{sp}(2)\oplus\mathfrak{sp}(1)$.
\columnbreak
$$\xymatrix{
\llap{$v_z\in$ }TP 
	\ar[dr]^{q_*} 
    \ar[dd]_{p_*} 
    \ar[r]^(.35){\omega} 
    \ar@/^2pc/[r]^{\omega_\fm}
    & \Lm^2 = \Lm^2_{15} \ar@/^2pc/[r]^{z\bullet} \oplus \fh & {\underline{\Lm^2_{15}}} \\ 
  &TN 
  	\ar[dl]_{\pi_*} 
    \ar@{=}[r]
    \ar[ur]_f 
    & \cV \rlap{ $\oplus\;\cH$}
  		\ar[u] ^\cI  \\ 
TM & &     
}$$
\end{multicols}
Since $\Sp(2)\Sp(1)\cong \Stab(\Om_0)$, it follows that there exists a \emph{universal section} $\Xi\in\Gamma(N,\pi^*(\Lm^4))$ defined by
\begin{align}
\label{eq: universal section}
    \Xi(y):=y^*\Om_0,
    \qforq
    y\in N.
\end{align}
Explicitly, one assigns to the class of frames $y\in N$ the vector of $\Lm^4(M)_{\pi(y)}$, the coordinates of which are given by $\Om_0$ in any frame $z_{\pi(y)}$. It follows that, to each homogeneous section $\sigma\in\Gamma(M,N)$, one can associate a geometric structure $\Om\in\Gamma(M,\Lm^4)$ modelled on $\Om_0$ by
\begin{align}
\label{eq: Xi correspondence}
    \Om_\sigma:=\sigma^*\Xi
    =\Xi\circ \sigma.
\end{align}
Conversely, to a given geometric structure $\Om\in\Gamma(M,\Lm^4)$ stabilised by $\Sp(2)\Sp(1)$, one associates, at $x\in M$, an $\Sp(2)\Sp(1)$-class of frames of $T_x M$ which, in turn, identifies an element of $\pi^{-1}(x)$, i.e. an element $\sigma (x)$ in the fibre of $\pi : N \to M$ over $x\in M$. Thus, we unambiguously  obtain $\Om$ from $\sigma$, and vice-versa.

Assuming that $M$ is compact, endowing the fibres of $N$ with the metric induced by the bi-invariant metric on $\SO(8)$, and  considering the metric induced by $g$ on $\mathcal{H}$, we define the energy functional
\begin{equation}
    E(\sigma) := \frac{1}{2}\int_M |d^{\mathcal{V}}\sigma|^2\ \vol,\label{energyofsection}
\end{equation}
where $d^{\mathcal{V}}$ denotes the projection of $d\sigma$ on $\mathcal{V}$. From the aforementioned discussion, we have that $\mathcal{I}(d^{\mathcal{V}}\sigma)=f(d\sigma)$. Note also that, since the horizontal space $\mathcal{H}$ is endowed with the metric $g$, the total energy of $\sigma$ is equal to (\ref{energyofsection}) and a constant multiple of the volume of $M$ cf. \cite{Loubeau2019}*{Lemma 3}. In what follows, we shall denote by $\nabla$ the associated Levi-Civita connection to the latter metric on $N$. 
The results in \cite{Loubeau2019} show that (\ref{energyofsection}) corresponds to the $L^2$-norm on the intrinsic torsion associated to $\Om_{\sigma}$: 
\begin{equation}
    E(\Om_{\sigma}) := \frac{1}{2}\int_M |T|^2\ \vol.\label{energyofomega}
\end{equation}
For the reader's convenience we summarise a few key relevant results in \cite{Loubeau2019}:
\begin{proposition}\label{summaryofresultsLSE}
The covariant derivative of the universal section $\Xi$ is given by
\begin{equation}
\label{definitionoff}
    \nabla_Y \Xi = f(Y) \diamond \Xi,
    \qforq
    Y\in TN.
\end{equation}
Harmonic $\Sp(2)\Sp(1)$-structures, defined as the critical points of (\ref{energyofsection}),  satisfy the Euler-Lagrange equation:
\begin{equation}
    \tau^{\mathcal{V}}(\sigma):=\mathrm{tr}_g(\nabla^{\mathcal{V}}d^{\mathcal{V}}\sigma) =0.
\end{equation}
Furthermore,
\begin{equation}
    \mathcal{I}((\nabla^{\mathcal{V}}d^{\mathcal{V}}\sigma)(X,X))=(\nabla^{\om} (\sigma^*f))(X,X),
\end{equation}
where $X\in TM$ and $\nabla^{\om}$ denotes the induced connection on $\pi^*\Lm^2$.
\end{proposition}
The harmonic section flow, defined as the negative gradient flow of (\ref{energyofsection}), starting from an $\Sp(2)\Sp(1)$-structure defined by $\sigma_0$, is given by
\begin{gather}
	\frac{d\sigma_t}{dt}=\tau^{\mathcal{V}}(\sigma_t),\label{harmonicmapflowequation},\\
	\sigma_0 = \sigma(0).
\end{gather}
Appealing to Proposition \ref{summaryofresultsLSE} and using the map $\mathcal{I}$, we can then reinterpret the above flow more concretely in terms of a geometric flow for $\Om_t$ (see \S\ref{sec: harmonicQKflow} below). This is the same procedure that leads to the harmonic flows of $\mathrm{G}_2$-, $\Spin(7)$- and $\rU(n)$-structures in \cites{Dwivedi2019, Grigorian2019, Dwivedi2021, HeLi2021}.

\subsection{Examples of harmonic quaternion-K\"ahler  structures}
\label{examplesofharmonicQKsection}

In this section we construct several explicit examples of strictly harmonic QK structures i.e. harmonic QK structures which are not torsion-free. We also refer the reader to \S\ref{solutiononsu3section} below for an example on the Lie group $\SU(3)$. The main result of this section is that conformally parallel $\Sp(2)\Sp(1)$-structures are harmonic. In fact our proof applies to $\mathrm{G}_2$- and $\Spin(7)$-structures as well, thereby extending the result of Grigorian in the $\rG_2$ case, cf. \cite{Grigorian2019}*{Theorem 4.3}. 

It is well-known that if $(M,g)$ a complete Einstein manifold, aside from the round sphere, then $g$ is the unique Einstein metric in its conformal class (up to homothetic rescaling), cf. \cite{Kuhnel1995}. In particular, if $g$ is a quaternion-K\"ahler metric, then $f^{2}g$ cannot be an Einstein metric unless $f$ is constant. So any harmonic structure in the class $[f^2g]$, if any exists, cannot be torsion-free. While it is natural to expect that the conformally rescaled metric will converge back under the Ricci flow to the Einstein metric (which is a Ricci soliton), the latter behaviour cannot happen in our case, since the harmonic flow preserves the metric. So this naturally motivates searching for harmonic structures in such conformal classes. 

\subsubsection{Example on a flat torus}\label{exampleonflattorus}

Consider the flat torus $\mathbb{T}^8$ with the quaternion-K\"ahler structure $\Om_0$ defined by expression (\ref{qk4form}). We define a conformally flat quaternion-K\"ahler structure by 
\begin{equation}
\label{conformalqkonT8}
    {\Om}:=f(x_1)^4 \Om_0,
\end{equation}
and denote the associated $\Sp(2)\Sp(1)$ coframing by $e^i:=f(x_1) dx_i$ and its dual by $e_i:=f(x_1)^{-1}\partial_{x_i}$. In terms of decomposition (\ref{eq: torsion decomposition}), we know that the torsion $T$ takes values in $\Om^1_{8}$, and it is essentially determined by the $1$-form $df$.
Indeed a direct calculation shows that the torsion tensor is explicitly given by
\begin{align*}
    T(e_1)&=0,\\
    T(e_2) &= \frac{1}{4} \frac{d}{dx_1}(f(x_1)^{-1})(e^{12}+e^{34}-e^{56}-e^{78}),\\
    T(e_3) &= \frac{1}{4} \frac{d}{dx_1}(f(x_1)^{-1})(e^{13}-e^{24}-e^{57}+e^{68}),\\
    T(e_4) &= \frac{1}{4} \frac{d}{dx_1}(f(x_1)^{-1})(e^{14}+e^{23}-e^{58}-e^{67}),\\
    T(e_5) &= \frac{1}{4} \frac{d}{dx_1}(f(x_1)^{-1})(3e^{15}-e^{26}-e^{37}-e^{48}),\\
    T(e_6) &= \frac{1}{4} \frac{d}{dx_1}(f(x_1)^{-1})(3e^{16}+e^{25}+e^{38}-e^{47}),\\
    T(e_7) &= \frac{1}{4} \frac{d}{dx_1}(f(x_1)^{-1})(3e^{17}-e^{28}+e^{35}+e^{46}),\\
    T(e_8) &= \frac{1}{4} \frac{d}{dx_1}(f(x_1)^{-1})(3e^{18}+e^{27}-e^{36}+e^{45}).
\end{align*}
and from this one finds that
\[\nabla_{e_i}(T(e_i))=0, \qforq i=1,\dots8.
\]
Another simple computation shows that
\[\nabla_{\nabla_{e_i}e_i} \Om =0.\]
Combining the above, we have
\[\mathrm{div}(T):=\sum_{i=1}^8\nabla_{e_i}(T(e_i))-T(\nabla_{e_i}e_i)=0.\]
Thus, we have just shown that ${\Om}$, as defined by (\ref{conformalqkonT8}), determines a harmonic section for the conformally flat metric $f(x_1)^2 g_{\mathbb{T}^8}$. Observe that $T=0$ if and only if $f(x_1)$ is constant, as expected.

\subsubsection{Example on the hyperbolic quaternionic plane}
\label{sec: example2}

Let us now consider the hyperbolic quaternionic plane $\mathcal{H}\mathbb{H}^2$. Topologically $\mathcal{H}\mathbb{H}^2$ is diffeomorphic to $\R^8$, but as a Riemannian manifold it is the symmetric space $\frac{\Sp(2,1)}{\Sp(2)\Sp(1)}$ i.e. the non-compact dual of $\mathbb{H}\mathbb{P}^2=\frac{\Sp(3)}{\Sp(2)\Sp(1)}$. In particular, it has holonomy group \textit{equal} to $\Sp(2)\Sp(1)$. As a cohomogeneity one manifold, under the action of the quaternion Heisenberg group, whose Lie algebra is given by
\[(0,0,0,0,12+23,13-24,14+23),\]
we can express the quaternion-K\"ahler metric as
\begin{equation}
g_{\mathcal{H}\mathbb{H}^2}=\frac{1}{16(1-s)^2}ds^2+\frac{1}{(1-s)}(\al_1^2+\al_2^2+\al_3^2)+\frac{1}{(1-s)^{1/2}}(dx_1^2+dx_2^2+dx_3^2+dx_4^2),\label{hyperbolicmetric}
\end{equation}
where the $1$-forms $\al_i$ are defined by
\begin{align*}
	\al_1 &= dx_5 - x_2 dx_1+ x_1 dx_2- x_4 dx_3+x_3 dx_4,\\
	\al_2 &= dx_6 - x_3 dx_1+ x_4 dx_2+ x_1 dx_3- x_2 dx_4,\\
	\al_3 &= dx_7 - x_4 dx_1- x_3 dx_2+ x_2 dx_3+ x_1 dx_4,
\end{align*}
and $s\in(-\infty,1)$, cf. \cites{UdhavFowdar4, Gibbons01}. As in the previous example, we consider conformal metrics given by $f(s)^2 g_{\mathcal{H}\mathbb{H}^2}$. We define an $\Sp(2)\Sp(1)$-coframing by setting 
$$
e^i=\frac{f(s)}{(1-s)^{1/4}} dx_{i}
\qforq
i=1,2,3,4,
$$ 
and 
$$
e^5=\frac{f(s)}{4(1-s)} dt
\qandq
e^{i+5}=\frac{f(s)}{(1-s)^{1/2}}\al_i
\qforq
i=1,2,3.
$$ 
One again verifies that
$\nabla_{\nabla_{e_i}e_i} \Om =0$, 
and that the torsion tensor is explicitly given by
\begin{align*}
	T(e_1) &=  (s-1)\frac{d}{ds}(f(s)^{-1})(3e^{15}-e^{26}-e^{37}-e^{48}),\\
	T(e_2) &= (s-1)\frac{d}{ds}(f(s)^{-1})(e^{16}+3e^{25}-e^{38}+e^{47}),\\
	T(e_3) &= (s-1)\frac{d}{ds}(f(s)^{-1})(e^{17}+e^{28}+3e^{35}-e^{46}),\\
	T(e_4) &= (s-1)\frac{d}{ds}(f(s)^{-1})(e^{18}-e^{27}+e^{36}+3e^{45}),\\
	T(e_5) &= 0,\\
	T(e_6) &= (s-1)\frac{d}{ds}(f(s)^{-1})(e^{12}+e^{34}-e^{56}-e^{78}),\\
	T(e_7) &= (s-1)\frac{d}{ds}(f(s)^{-1})(e^{13}-e^{24}-e^{57}+e^{68}),\\
	T(e_8) &= (s-1)\frac{d}{ds}(f(s)^{-1})(e^{14}+e^{23}-e^{58}-e^{67}).
\end{align*}
As above we find that $\nabla_{e_i}(T(e_i))=0$, for each $i$, and hence
\[\mathrm{div}(T)=0.\]
Thus, the above conformally parallel quaternion-K\"ahler structures indeed define harmonic sections.

The above examples seem to suggest that conformally parallel $\Sp(2)\Sp(1)$-structures might always be harmonic; in fact this is known to be true in the $\rG_2$ case, cf. \cite{Grigorian2019}*{Theorem 4.3}. 
We shall now show that this holds for a larger class of $H$-structures, including $\Sp(2)\Sp(1)$ and $\Spin(7)$. 

\begin{proposition}\label{torsionofconformallyparallelproposition}
Let $(M,g,\xi)$ be a Riemannian manifold with holonomy group contained in $H\subset \SO(n)$, where $\xi$ is a parallel $k$-form which determines the holonomy reduction. Consider the conformal data given by $\tilde{g}=e^{2f}g$ and $\tilde{\xi}=e^{kf}\xi$ on $M$. The intrinsic torsion $\tilde{T}$, defined by 
\begin{equation}
    \tilde{\nabla}_X \tilde{\xi}= \tilde{T}(X)\  \tilde{\diamond}\ \tilde{\xi},
\end{equation}
satisfies
\begin{align*}
(\tilde{\nabla}_Y \tilde{T})(X)\ =\ &\pi_{\mathfrak{m}}(\tilde{X}^\flat \w \nabla_Y df+ g(Y,\nabla f) \tilde{X}^\flat \w df - g(X,\nabla f) \tilde{Y}^\flat \w df)\\
&+(\tilde{Y}^\flat \w df)\ \tilde{\diamond}\ \pi_{\mathfrak{m}} (\tilde{X}^\flat \w df) - 2g (Y,\nabla f) \pi_{\mathfrak{m}}(\tilde{X}^\flat \w df),
\end{align*}
where $\pi_{\mathfrak{m}}$ denotes the orthogonal projection in $\Lm^2\cong \mathfrak{h}\oplus\mathfrak{m}$ cf. \S\ref{sec: harmonicsp2sp1sections} and \cite{Loubeau2019}*{Part I}. 
\end{proposition}
\begin{proof}
    Let $\tilde{\nabla}$ (respectively $\nabla$) denote the Levi-Civita connection of $\tilde{g}$ (respectively $g$). For any $p$-form $\al$, we have
\begin{equation}
\label{conformalidentity}
    \tilde{\nabla}_X\al = \nabla_X\al+(X^\flat \w df) \diamond \al - pg(X,\nabla f)\al.
\end{equation}
Applying the above to $\al=\tilde{\xi}$ and using the fact that $\nabla \xi=0$, we have that 
\begin{equation}
    \tilde{\nabla}_X \tilde{\xi}= (\tilde{X}^\flat \w df)\  \tilde{\diamond}\ \tilde{\xi},
\end{equation}
where we decorate with $\tilde{\ }$ quantities defined with respect to $\tilde{g}$, and we also used that $(X^\flat \w df) \diamond \al = (\tilde{X}^\flat \w df)\ \tilde{\diamond}\ \al$. It now follows that the intrinsic torsion of $\tilde{\xi}$ is given by 
\begin{equation}
    \tilde{T}(X) = \pi_{\mathfrak{m}}(\tilde{X}^\flat \w df).
\end{equation}
It is worth pointing out that the projection map $\pi_{\mathfrak{m}}$ only depends on the conformal class of $\xi$, so there is no ambiguity here. Moreover, such a formula for the intrinsic torsion was indeed to be expected, because $\tilde{T}$ vanishes if and only if $f$ is constant, so $\tilde{T}$ has to correspond to some pairing between $\tilde{g}$ and $df$. The reader can also observe this in the explicit examples given in \S\S\ref{exampleonflattorus}, \ref{sec: example2}.

Since $g$ has holonomy contained in $H$, it follows that $\nabla$ and $\pi_{\mathfrak{m}}$ commute. Hence applying (\ref{conformalidentity}) again, with $\al=\pi_{\mathfrak{m}}(\tilde{X}^\flat \w df)$, gives
\begin{equation*}
    \tilde{\nabla}_Y(\pi_{\mathfrak{m}}(\tilde{X}^\flat \w df)) =\pi_{\mathfrak{m}}(\nabla_Y\tilde{X}^\flat \w df+\tilde{X}^\flat \w \nabla_Y df)
    +(\tilde{Y}^\flat \w df)\ \tilde{\diamond}\ \pi_{\mathfrak{m}} (\tilde{X}^\flat \w df) - 2g (Y,\nabla f) \pi^2_{15}(\tilde{X}^\flat \w df).
\end{equation*}
Using (\ref{conformalidentity}) yet again, with $\al=\tilde{X}^\flat$, and substituting in the above yields the result.
\end{proof}

\begin{theorem}
\label{conformalimpliesharmonictheorem}
    Let $(M,\tilde{g},\tilde{\Om})$ be a locally conformally parallel quaternion-K\"ahler manifold i.e. there exists locally a function $f$ such that $\Om=e^{-4f}\tilde{\Om}$ defines a (local) torsion-free quaternion-K\"ahler structure on $M$. Then $\tilde{\Om}$ defines a harmonic $\Sp(2)\Sp(1)$-structure.
\end{theorem}
\begin{proof}
    Taking $\xi=\Om$ in Proposition \ref{torsionofconformallyparallelproposition}, we have
\begin{align}
    (\tilde{\nabla}_Y \tilde{T})(X)
    =\ & \pi^2_{15}(\tilde{X}^\flat \w \nabla_Y df+ g(Y,\nabla f) \tilde{X}^\flat \w df - g(X,\nabla f) \tilde{Y}^\flat \w df)\nonumber\\
    &+(\tilde{Y}^\flat \w df)\ \tilde{\diamond}\ \pi^2_{15} (\tilde{X}^\flat \w df) - 2g (Y,\nabla f) \pi^2_{15}(\tilde{X}^\flat \w df).
    \label{divergenceofconformal}
\end{align}
    Since $\mathrm{skew}(\nabla df) = d^2f=0$, we see that $\sum_{i=1}^8\pi^2_{15}(\tilde{E_i}^\flat \w \nabla_{E_i} df)=0$, where $E_i$ denotes a local orthonormal framing with respect to $g$. Working at a point and identifying $\Om$ with $\Om_0$ and $E_i$ with $\partial_{x_i}$, it suffices to check directly  that the last two terms in (\ref{divergenceofconformal}) vanish as well, when summing over the $E_i$. Since $\Sp(2)\Sp(1)$ acts transitively on the unit sphere, we can also identify $\nabla f$ with $c \partial_{x_1}$ at a point, to further ease computation. In those terms it is straightforward to check that $\mathrm{div}_{\tilde{g}}(\tilde{T})=0$.
\end{proof}
The argument in the above proof can also be applied to the $\rG_2$ and $\Spin(7)$ cases. This reduces the problem of computing the divergence of torsion for a conformally parallel structure to verifying that that the last two terms in (\ref{divergenceofconformal}) vanish when summing over $i$, which is essentially just a pointwise computation. The result in the $\rG_2$ case is already known, cf. \cite{Grigorian2019}*{Theorem 4.3}), and our argument extends easily to the $\Spin(7)$ case:
\begin{corollary}
    A locally conformally parallel $\Spin(7)$-structure on an $8$-manifold is harmonic.
\end{corollary}
\begin{proof}
    Repeat the proof of Theorem \ref{conformalimpliesharmonictheorem} with $\Om$ replaced by the $\Spin(7)$-structure $4$-form $\Phi$, which is pointwise modelled on  (\ref{eq: spin7form}).
\end{proof}

\section{Quaternion-K\"ahler harmonic flow: basic properties}
\label{sec: harmonicQKflow}

In this section we derive the harmonic flow equation for   quaternion-K\"ahler (QK) structures and define the corresponding notion of soliton. 

We can express the harmonic flow (\ref{harmonicmapflowequation}) for $H=\Sp(2)\Sp(1)$ in terms of an evolution equation for the defining $4$-form $\Om(t)$.  First, in terms of the isomorphism \eqref{eq: isomorphism I:V->m} between the vertical component of $TN$ and the bundle $\ufm$, we have
\begin{align*}
	\mathcal{I}(\tau^{\mathcal{V}}(\sigma_t)) 
	&= \mathcal{I}(\tr (\nabla^\mathcal{V} d^\mathcal{V} \sigma_t)) 
	= \tr (\nabla (\sigma_t^*f))
	= \tr (\nabla T_t)\\
	&=   \mathrm{div}\ T_t
\end{align*}
where we used Proposition \ref{summaryofresultsLSE} and expression (\ref{eq: qktorsion}) for the intrinsic torsion.
Extending the connection form $f$ to $M^8 \times \R_t$, and performing the same computation as above, we get 
\begin{equation}
	\mathcal{I}\Big(\frac{d\sigma_t}{dt}\Big) =  f(d\sigma_t(\partial_t))= \frac{1}{32}  \frac{d\Om_t}{dt} \ip_3 \Om_t,
\end{equation}
where we inverted the diamond operator $\diamond$ in (\ref{definitionoff}) using the operator $\ip_3$.

In view of the results of the previous section, the harmonic flow of a $\Sp(2)\Sp(1)$-structure starting at $\Om_0$ becomes:
\begin{gather}
\left\{
\begin{array}{rcl}
    \displaystyle
    \frac{d \Om}{d t} &=& (\mathrm{div}\ T) \diamond \Om
    \label{harmonicQKflowequation}\\
    \Om(0) &=&\Om_0
\end{array}
\right..
\end{gather}
We shall also refer to the above flow as the \emph{QK harmonic flow}.
As an instance of the general theory of harmonic $H$-flows, we already know that the flow admits a unique short-time solution, given smooth initial data. Moreover, if the flow exists for a maximal time $T_{max}$, then  $\sup_{x\in M}|T_t| \to \infty$ as $t \to T_{max}$  \cite{Loubeau2019}*{Theorems 1 and 2}. In this section we will develop the technical results necessary to study the behaviour of the flow as $t\to T_{max}$, and investigate under what circumstances $T_{max}$ can be extended to infinity.

\subsection{Evolution of the  intrinsic torsion}

Let us derive the evolution of the torsion tensor $T$ under the harmonic flow (\ref{harmonicQKflowequation}).
\begin{proposition}
\label{evolutionoftorsionproposition}
    Under the harmonic flow (\ref{harmonicQKflowequation}) the intrinsic torsion $T$ evolves by
\begin{equation}
\label{evolutionoftorsionequation}
    \frac{\partial T}{\partial t}(X)= \nabla_X(\mathrm{div}\ T)-\frac{1}{32}\Big((\mathrm{div}\ T \diamond \Om)\ip_3 (T(X)\diamond \Om )-
    (T(X)\diamond \Om ) \ip_3
    (\mathrm{div}\ T \diamond \Om)\Big),
\end{equation}
for $X\in \Gamma(TM)$. Moreover, 
\begin{equation}
    \pi^{2}_{15}\Big(\frac{\partial T}{\partial t}(X) \Big) 
    = \pi^{2}_{15}\Big(\nabla_X(\mathrm{div}\ T)\Big).
\end{equation}
\end{proposition}
\begin{proof}
    From (\ref{eq: qktorsion}), we have 
\begin{align}
    \frac{\partial T}{\partial t}(X) 
    &= \frac{1}{32} (\nabla_X\Big(\frac{\partial \Om}{\partial t}\Big)\ip_3 \Om+ (\nabla_X \Om)\ip_3 \Big(\frac{\partial \Om}{\partial t}\Big)),\nonumber\\
    &= \frac{1}{32} (\nabla_X\Big(\mathrm{div}\ T \diamond \Om\Big)\ip_3 \Om+ (\nabla_X \Om)\ip_3 \Big(\mathrm{div}\ T \diamond \Om\Big)),
\end{align}
    where for the first equality we used the fact that $g$ is unchanged along the flow and hence so are $\nabla$ and $\ip_3$, and for the second equality we use (\ref{harmonicQKflowequation}). The first part of the proposition now follows from 
\[32\nabla_X(\mathrm{div}\ T)
=\nabla_X((\mathrm{div}\ T \diamond \Om) \ip_3 \Om )= \nabla_X\Big(\mathrm{div}\ T \diamond \Om\Big)\ip_3 \Om+(\mathrm{div}\ T \diamond \Om)\ip_3 (T(X) \diamond \Om) ,\]
    where we again use the fact that $\ip_3$ only depends on $g$ and hence is invariant under $\nabla$.
For the second part we use the same argument as in the proof of Proposition \ref{bianchiidentityproposition} i.e.
\[(\mathrm{div}\ T \diamond \Om)\ip_3 (T(X)\diamond \Om ) \in \Lm^{4+}_{15} \otimes \Lm^{4+}_{15} \cong (\R \oplus S^2_0(\Lm^2_0 E) \oplus S^2 E)\otimes (\R \oplus S^4 H \oplus S^2 H) \]
and as such it has no component in $\Om^2_{15}$; likewise for $(T(X)\diamond \Om ) \ip_3
    (\mathrm{div}\ T \diamond \Om)$. This concludes the proof.
\end{proof}
The evolutions of Dirichlet energy and density now follow immediately:
\begin{corollary}
\label{evolutionoftorsionsquare}
    The norm square of $T$ evolves by
\begin{equation}
    \frac{\partial |T|^2}{\partial t} 
    := \frac{\partial}{\partial t}(g(T,T)) 
    = 2 g(\nabla \mathrm{div}T,T).
\end{equation}
    In particular, 
\begin{equation}
\label{evolutionoftorsionsquareequation}                    \frac{\partial}{\partial t}\int_M |T|^2 \ \vol 
    = -2 \int_M |\mathrm{div}T|^2\ \vol.
\end{equation}
\end{corollary}
Note that (\ref{evolutionoftorsionsquareequation}) was to be expected, since the harmonic flow is just the negative gradient flow of the energy functional. Furthermore, however much the $L^2$-norm of $T$ decreases under the flow, it can still concentrate over certain points on $M$, thereby resulting in singularities. In order to analyse such behaviour, we need a monotonicity formula, which will derive in \S\ref{almostmonotonicitysubsection}. Next we show that the harmonic flow admits a parabolic rescaling.

\subsection{Parabolic rescaling}

In the study of geometric flows one often encounters finite-time singularities. These singularities are in many cases modelled on soliton solutions to the flow, and hence classifying those becomes an important problem. To find these solitons as one approaches a singularity, one performs a parabolic scaling i.e. a rescaling of geodesic distance by $x \to c x$ while time scales by $t \to c^2 t$, for some constant $c$. Provided that we have a compactness theorem, this allows one to take a suitable limit of the flow and thus to extract information about the singularity; this procedure is well-known for eg. for the Ricci and mean curvature flows. 
To perform the analogous scaling in our context, we first consider the behaviour of the intrinsic torsion under a homothetic rescaling, in accordance with the homogeneity degree of the QK $4$-form.

\begin{lemma}
\label{lem: paraboliclemma}
	Under the homothetic transformation $\tilde{\Om}:=c^4 \Om$, the torsion form transforms as $\tilde{T}=c^2 T$ and hence $\mathrm{div}_{\tilde{g}}(\tilde{T})= \mathrm{div}_g(T).$
\end{lemma}
\begin{proof}
	Observe that the homothetically rescaled metric is given by $\tilde{g}_{\tilde{\Om}}=c^2g_\Om$, while the Levi-Civita connection remains unchanged i.e.  $\tilde{\nabla}=\nabla$. Thus, we compute
	\begin{align*}
		\tilde{T}(X) \ \tilde{\diamond}\ \tilde{\Om} &=  c^4\ \nabla_X \Om 
		=c^4\ T(X) \diamond \Om \\
		&= c^2\ T(X)\ \tilde{\diamond}\ \tilde{\Om},
	\end{align*}
    where $\tilde{\diamond}$ denotes the associated operator to $\tilde{g}_{\tilde{\Om}}$. It is worth emphasising that by definition the operator $\diamond$ acting on $2$-forms depends on the metric $g$.
\end{proof}

\begin{corollary}
\label{corollaryparabolic}
	If $\Om_t$ is a solution to (\ref{harmonicQKflowequation}) defined for $t \in [0,T_{max})$ then under the parabolic rescaling $(\Om_t,t)\to (\tilde{\Om}_{\tilde{t}}:=c^4\Om, \tilde{t}:=c^2 t)$, $\tilde{\Om}_{\tilde{t}}$ is again a solution to (\ref{harmonicQKflowequation}) but now defined for $\tilde{t}\in[0,c^2T_{max})$.
\end{corollary}
\begin{proof}
	We compute directly $\frac{d \tilde{\Om}_{\tilde{t}}}{d \tilde{t}} 
	=  c^2 \mathrm{div}_g(T) \diamond \Om 
	= \mathrm{div}_{\tilde{g}}(\tilde{T})\ \tilde{\diamond}\ \tilde{\Om}$,
    using (\ref{harmonicQKflowequation}) for the first equality and Lemma \ref{lem: paraboliclemma} for the second one.
\end{proof}
Next we introduce the notion of solitons for the harmonic QK flow (\ref{harmonicQKflowequation}).

\subsection{Harmonic \texorpdfstring{$\Sp(2)\Sp(1)$}{} solitons}

The simplest solutions to a geometric flow are those that evolve by scaling symmetry of the flow equation; these are called solitons and arise naturally when analysing singularities of the flow (see Theorem \ref{type1singularitymodel} below). We now describe what harmonic solitons look like in our context. 

\begin{definition}
\label{definitionSelfsimilarsolution} 
    A solution $\{\Om(t)\}$ of the harmonic QK flow \eqref{harmonicQKflowequation} is said to be  \emph{self-similar} if there exist a function $\rho(t)$, with $\rho(0)=1$, and a family of diffeomorphisms $\{f(t):M\rightarrow M\}$, with $f(0)=\mathrm{Id}$, such that
\begin{equation}
\Om(t) = \rho(t)^{4}{f(t)}^*\Om_0, \quad\forall\  t\in \, [0,T_{max}). \label{QKselfsimilarsolution}
\end{equation} 
\end{definition}
We shall now justify the name \textit{self-similar solution}.
Denoting by $W(t)\subset\sX(M)$ the infinitesimal generator of $f(t)\subset\Diff{(M)}$, the \emph{stationary vector field} of a self-similar solution is defined by
\begin{equation}
\label{stationaryvectorfieldQK}
    X(t):=(f(t)^{-1})_*{W(t)}\in \sX(M),
    \quad\forall\  t\in \, [0,T_{max}).
\end{equation}
From (\ref{QKselfsimilarsolution}) we immediately deduce that the metric evolves by
\begin{equation}
    g(t)=\rho(t)^2 f(t)^*g_0.
\end{equation}
On the other hand, since the harmonic flow is isometric, i.e. its time-derivative $g'(t)$ vanishes,
\begin{equation}
\label{solitonmetric}    \mathcal{L}_{X(t)}g_0
    =-2(\log{\rho(t)})'g_0.
\end{equation}
In particular, this shows that $\rho(t)$ completely determines $f(t)$ (up to isometry). 
Furthermore, specialising  \cite{Dwivedi2021}*{Lemma 2.9} to the case $H=\Sp(2)\Sp(1)$, we know that the torsion tensor $T(t)$ of $\Om(t)$ satisfies
\begin{equation*}
	\mathrm{div}\ T(t) = X(t) \ip T(t) + \frac{1}{2}\ \pi^2_{15} (dX(t)^\flat).
\end{equation*}
The above can be also shown quite easily using (\ref{harmonicQKflowequation}) and (\ref{QKselfsimilarsolution}). We should emphasise that the projection map $\pi^2_{15}:\Lm^2 \to \Lm^2_{15}$ is also time-dependent, since it is determined by $\Om(t)$.
The above motivates the following definition. 
\begin{definition}\label{definitionofsoliton}
A harmonic $\Sp(2)\Sp(1)$-soliton on a Riemannian manifold $(M^8,g)$ is given by a triple $(\Om, X, c)$, where $\Om$ induces the metric $g$, $X$ is a vector field and $c$ is a constant such that
 \begin{gather}
\left\{
\begin{array}{rcl}
    \displaystyle
    \mathcal{L}_X g &=& c g,\label{QKsolitonequation}\\
    \mathrm{div}\ T &=& X \ip T + \frac{1}{2}\ \pi^2_{15} (dX^\flat).
\end{array}
\right.
\end{gather}
    According to whether  $c<0$, $c=0$ or $c>0$,  the corresponding soliton is said to be  \emph{shrinking}, \emph{steady} or \emph{expanding}, respectively.
\end{definition}
We can now show that solitons indeed give rise to self-solution solutions of (\ref{harmonicQKflowequation}).
\begin{proposition}
    A harmonic $\Sp(2)\Sp(1)$-soliton, as in (\ref{QKsolitonequation}), induces a self-similar solution.
\end{proposition}
\begin{proof}
    We first consider the case $c\neq 0$. Let $\rho(t)=(t+1)^{-c/2}$ and $X(t)=(t+1)^{-1}{X}$, so that $\rho(0)=1$ and $X(0)={X}$. It is easy to see that this satisfies (\ref{solitonmetric}). We then define $f(t$ by
\[\frac{d}{dt}f(t)=X(t) f(t),
    \qwithq f(0)=\mathrm{Id}.
\]
    Applying \cite{Dwivedi2021}*{Lemma 2.6},  to the $H=\Sp(2)\Sp(1)$ case, we have
\[\mathcal{L}_{X(t)}\Om
=(X(t) \ip T +\frac{1}{2}\mathcal{L}_{X(t)}g +\frac{1}{2}\pi^2_{15}(dX(t)^\flat))\diamond \Om\]
    and one easily checks from the definition of $\diamond$, see Remark \ref{defintionofdiamond},
    that $g \diamond \Om = 4 \Om$. 
    Defining $\Om(t)$ by (\ref{QKselfsimilarsolution}), and using the above together with (\ref{QKsolitonequation}), one verifies directly that this indeed defines a solution to (\ref{harmonicQKflowequation}). Note that the resulting expanding and shrinking solitions are defined for $t\in (-1,\infty)$.
    If $c=0$, then we can take $\rho(t)=1$ and $X(t)=X$, obtaining an eternal steady soliton solution.
\end{proof}

Note that a soliton does not determine a unique self-similar solution, in fact for any function $h$, depending only on $t$, such that $h(0)=0$ and $h'(0)=-\frac{c}{2}$, we can set
$$
    \rho = \exp(h(t))
    \qandq 
    X(t) = -\frac{2}{c} h'(t)X.
$$
For instance, setting $h(t)=-\frac{c}{2}t$ we get eternal skrinkers and expanders as well. 

It is a classical result that the only complete Riemannian manifold with a non-Killing homothetic vector field is Euclidean space, cf. \cite{Tashiro1965}. We immediately deduce that:
\begin{corollary}\label{cor: shrinkersandexpandeers}
    Shrinking and expanding solitons of the QK harmonic flow (\ref{harmonicmapflowequation}) are always isometric to Euclidean $\R^8$.
\end{corollary}

Note that there are plenty of non-parallel structures $\Om$ on $\R^8$ inducing the Euclidean metric, so 
it natural to ask whether there exists any non-torsion-free shrinking or expanding soliton.  We shall answer in the affirmative with an explicit example of a steady soliton in \S\ref{sec: steadysolitonexample}, by means of the following simple idea. If $X=\nabla f$ is some gradient vector field, then $d X^{\flat}=0$ and hence  \textit{gradient harmonic solitons} satisfy 
\begin{equation}
\label{gradientsoliton}
    \mathrm{div}\ T = T(\nabla f).
\end{equation}

\section{Quaternion-K\"ahler harmonic flow: long-time existence and singularities}
\label{sec: long-time and sing}

By exploiting the similarities with the harmonic flow of $\Spin(7)$-structures, we readily obtain a compactness theorem for the harmonic quaternion-K\"ahler flow. We also prove an almost-monotonicity formula, by building upon the recent work in \cite{HeLi2021} in the context of almost-Hermitian structures. Our monotonicity formula also applies to the $\Spin(7)$ case and hence leads to a stronger convergence result than in \cite{Dwivedi2021}. In fact our proof of the monotonicity formula extends to a much more general class of $H$-structures, as established independently in  \cite{Fadel2022}.
In the last part we describe the singular set of the flow.

\subsection{Compactness}
If a solution to the harmonic flow  (\ref{harmonicQKflowequation}) has a finite-time singularity, then we obtain a new sequence of solutions by performing parabolic rescalings. In order to analyse the singularity, we need to be able to take a limit of such a sequence, following the standard method used for instance for the Ricci flow and mean curvature flow. We begin by specifying the notion of limit in our context:
\begin{definition}  
    Let $(M_i^8,\Om_i, g_i)$ be a sequence of complete Riemannian manifolds, with $\Sp(2)\Sp(1)$-structures determined by $\Om_i$ and marked points $p_i\in M_i^8$. Then we call $(M^8,\Om,p)$ a limit of the sequence, and write
$$
    (M_i^8,\Om_i,p_i) \to (M^8,\Om,p),
$$
    if there exists a sequence of compact sets $ \{U_i\}$ exhausting $M^8$ with $p_i\in \mathrm{int}(U_i)$, and a sequence of diffeomorphisms $\{F_i:U_i \to F(U_i)\subset M_i^8\}$ with $F_i(p)=p_i$, such that, on every compact set $K \subset M^8$ and for each $\varepsilon>0$, there exixts $i_0$ (depending on $\varepsilon$) such that 
\[\sup_{x\in K}|\nabla^k(F^*_i\Om_i-\Om)|_{g_0}<\varepsilon,
\quad\forall i\geq i_0,\]
    where $g_0$ denotes a fixed reference metric on $M^8$ and $\nabla$ is its corresponding Levi-Civita connection.
\end{definition}

We can now state the compactness theorem for the QK harmonic flow.
\begin{theorem}
\label{thm: compactness}
    Let $M_i$ be a sequence of compact $8$-manifolds with marked points $p_i\in M_i$, and let $\{\Om_i(t)\}$ denote a sequence of solutions to the QK harmonic flow (\ref{harmonicQKflowequation}) on $M_i$
    defined for $t\in (a,b)$. Suppose the following assumptions hold:
\begin{gather}
    \sup_i\sup_{M_i \times  (a,b)} |T_i(x,t)|_{g_i} < \infty,
    \label{compactnesshypothesisone}\\
    \inf_i \mathrm{inj}(M_i^8,g_i(0),p_i)>0,
    \label{compactnesshypothesistwo}\\
    \sup_i|\nabla^kR_i|\leq C_k,
    \label{compactnesshypothesisthree}
\end{gather}
where $\mathrm{inj}$ denotes the injectivity radius and $C_k$ are uniform constants independent of $i$. Then there exist a manifold $M^8$, with a marked point $p\in M$, and a solution $\Om(t)$ to (\ref{harmonicQKflowequation}) on $M$, defined for $t\in (a,b)$, arising as the  subsequential limit 
\[
    (M_i,\Om_i(t),p_i) \to (M,\Om(t),p),
    \quad\text{as}\quad 
    i \to \infty.
\]
\end{theorem}
\begin{proof}
Since the proof is analogous to the $\rG_2$ and $\Spin(7)$ cases, we shall only highlight the key parts of the argument, referring the reader to \cite{Dwivedi2019}*{Theorem 3.13} and \cite{Dwivedi2021}*{Theorem 4.19} for further detail.

In order to obtain the limit space $(M,g,p)$ as a complete pointed Riemannian manifold, we resort to the Cheeger-Gromov compactness theorem, cf. \cite{Hamilton1995}*{Theorem 2.3}. This relies on hypotheses \eqref{compactnesshypothesistwo} and \eqref{compactnesshypothesisthree}: condition \eqref{compactnesshypothesistwo} ensures that collapsing/degeneration-type phenomena does not occur, and \eqref{compactnesshypothesisthree} ensures that curvature does not concentrate along the limiting process. More precisely, Cheeger-Gromov compactness gives rise to an exhausting family of compact nested sets $U_i \subset M^8$ and diffeomorphisms $F_i:U_i \to F_i(U_i) \subset M_i^8$, such that $g =\lim F_i^*g_i$. Note that here we are also using the fact that the metric is unchanged under the flow, i.e. $g_i(0)=g_i(t)$.

Next we need to obtain the limit $4$-form $\Om(t)$, which, by contrast to the metric, does vary with time. In view of the  Shi-type estimates for the general harmonic flow of $H$-structures  \cite{Dwivedi2021}*{Proposition 2.16}, 
Assumption \eqref{compactnesshypothesisone} guarantees uniform bounds on all derivatives of torsion, for all time $t \in (a,b)$. Appealing to the Arzel{\'a}-Ascoli theorem, we can extract a $4$-form $\Om(t)$ on $M$ as the limit of $F_i^*\Om_i(t)$. Now, it is not a priori clear that the limit $\Om(t)$ also defines an $\Sp(2)\Sp(1)$-structure on $M$. 
To deduce to latter, we take the limit of \eqref{eq: qkmetricfromOm} and use the fact that the Riemannian metric $g$ arises as the Cheeger-Gromov limit of $(g_i)$. This concludes the proof.
\end{proof}

\subsection{The almost-monotonicity formula}
\label{almostmonotonicitysubsection}

This section is strongly based on the celebrated methods developed by Hamilton in \cite{Hamilton1993}, which we invite the unfamiliar reader to consult. 
Let $(M^8,g)$ be a complete Riemannian manifold. For $p \in M^8$, we denote by $u_{(p,t_0)}$ a positive fundamental solution of the backward heat equation, starting from the delta function at $p$, at time $t_0$, i.e.
\[\Big(\frac{\partial}{\partial t}+\Delta\Big)u_{(p,t_0)}=0, \ \ \ \lim_{t \to t_0} u_{(p,t_0)}(t)= \delta_{p},
\]
and we set 
$$u_{(p,t_0)}
=\frac{\exp\{-f_{(p,t_0)}\}}{\big(4 \pi(t_0-t)\big)^4}. 
$$  
In what follows we shall simply write $u=u_{(p,t_0)}$.
Suppose now that we have a solution to the harmonic QK flow (\ref{harmonicQKflowequation}) on $(M^8,g)$, defined for $t\in[0,t_0)$. Then, following \cite{GraysonHamilton1996}, we define the functional
\begin{equation}
\label{ThetaFunctional}            \Theta_{(p,t_0)}(\Om(t)) 
    := (t_0-t)\int_M u(t) |T(t)|^2\ \vol,
\end{equation}
which is invariant under parabolic rescaling -- unlike the energy functional $E$. 
Our first goal is to prove that $\Theta$ satisfies an almost-monotonicity formula. We begin by proving the following key lemma.
\begin{lemma}
\label{preAlmostMonotonicityLemma}
    Under the harmonic QK flow (\ref{harmonicQKflowequation}), the functional $\Theta$ evolves by
\begin{align}
    \frac{\partial}{\partial t}\Theta 
    = &-2(t_0-t) \int_M u |\mathrm{div} T-T(\nabla f)|^2\ \vol\label{preAlmostMonotonicityEquation}\\ 
    &-2(t_0-t) \sum_{i=1}^8 \int_M g( T(\nabla_{E_i}\nabla u),T(E_i))\ \vol +2(t_0-t) \int_M \frac{|T(\nabla u)|^2}{u}\ \vol- \int_M u |T|^2\ \vol\nonumber\\ 
    &+2(t_0-t) \sum_{i=1}^8 \int_M g(R(\nabla u,E_i),T(E_i))\ \vol,\nonumber
\end{align}
where $\{E_i\}$ denotes a local orthonormal framing.
\end{lemma}
\begin{proof}
Using the definition of $u$, a direct calculation shows that
\[\frac{\partial}{\partial t}\Theta = \int_M - u |T|^2+(t_0-t)  u  \frac{\partial |T|^2}{\partial t}-(t_0-t)  \Delta u  |T|^2\ \vol.\]
Let us consider the last summand in the above expression. 
Integrating by parts we have
\begin{align*}
    \int_M \Delta u |T|^2\ \vol 
    &= - \int_M g( \nabla u,\nabla |T|^2)\ \vol
    = -2 \int_M g( \nabla_{\nabla u} T,T)\ \vol\\
    &= -2 \sum_{i=1}^8 \int_M g( (\nabla_{\nabla u} T)(E_i),T(E_i))\ \vol\\
    &= -2 \sum_{i=1}^8 \int_M g( (\nabla_{E_i} T)(\nabla u)+R(\nabla u,E_i),T(E_i))\ \vol\\
    &= +2\int_M g(\mathrm{div} T, T(\nabla u))\ \vol -2 \sum_{i=1}^8 \int_M g( -T(\nabla_{E_i}\nabla u)+R(\nabla u,E_i),T(E_i))\ \vol
\end{align*}
    where we used the Bianchi type identity of Proposition \ref{bianchiidentityproposition} and the fact that $T(E_i)\in \Om^2_{15}$ for the penultimate equality.
    Another integration by parts, together with Corollary \ref{evolutionoftorsionsquare}, shows that
\begin{equation}
    \int_M u  \frac{\partial |T|^2}{\partial t} \ \vol 
    = -2 \int_M u |\mathrm{div}T|^2 +g(\mathrm{div}T, T(\nabla u))\ \vol.
\end{equation}
    Combining the above, we have
\begin{align*}
    \frac{\partial}{\partial t}\Theta 
    = &- \int_M u |T|^2\ \vol -2(t_0-t) \int_M u |\mathrm{div}T|^2 +2g(\mathrm{div}T, T(\nabla u))\ \vol\\ 
    &-2(t_0-t) \sum_{i=1}^8 \int_M g( T(\nabla_{E_i}\nabla u)-R(\nabla u,E_i),T(E_i))\ \vol\\
    = &- \int_M u\cdot |T|^2\ \vol -2(t_0-t) \int_M u |\mathrm{div} T-T(\nabla f)|^2-\frac{|T(\nabla u)|^2}{u}\ \vol\\ 
    &-2(t_0-t) \sum_{i=1}^8 \int_M g( T(\nabla_{E_i}\nabla u)-R(\nabla u,E_i),T(E_i))\ \vol.
    \qedhere
\end{align*}
\end{proof}

Equipped with the above lemma, we now apply the same argument as in \cite{Dwivedi2021}*{Theorem 5.2} to obtain the following monotonicity result: 

\begin{theorem}[Weak almost-monotonicity formula]
\label{almostmonotonicitytheorem}
    Let $\{\Om(t)\}$ be a solution of the harmonic QK flow (\ref{harmonicQKflowequation}) on $(M,g)$, and let  $0<\tau_1<\tau_2<t_0$. The following assertions hold:
\begin{enumerate}
    \item If $M$ is compact, then there exist constants $K_1,K_2>0$, depending only on the geometry of $(M,g)$, such that 
\begin{equation}
    \Theta(\Om(\tau_2)) 
    \leq K_1 \Theta(\Om(\tau_1)) + K_2(\tau_2-\tau_1) (E(\Om(0))+1).
\end{equation}
    
    \item If $M=\R^8$ with its Euclidean structure, then
\begin{equation}
    \Theta(\Om(\tau_2)) 
    \leq \Theta(\Om(\tau_1)).
\end{equation}
\end{enumerate}
\end{theorem}
\begin{proof}
    The proof follows Hamilton's original argument, which also appears in detail in \cites{Dwivedi2019, Dwivedi2021}, so we shall only outline the key steps. It is worth emphasising that, although we consider here the structure group $H=\Sp(2)\Sp(1)$, whereas \cite{Dwivedi2019} considers $H=\rG_2$ and \cite{Dwivedi2021} considers $H=\Spin(7)$, the argument is essentially the same with $T \in \Gamma(\Lm^1 \otimes \mathfrak{h}^{\perp}).$ In other words, the proof is independent of the structure group $H \subset \SO(n)$, so long as a Bianchi-type identity holds. We illustrate this below by avoiding multi-index computations specific to some choice of $H$. 

    Integrating by parts the last term of (\ref{preAlmostMonotonicityEquation}),  
\begin{equation}
\label{curvatureterm}
    2(t_0-t) \sum_{i=1}^8 \int_M g(R(\nabla u,E_i),T(E_i))
\end{equation} 
    and using again (\ref{eq: bianchi identity}) gives an integral involving $u$, $T$, $R$ and $\nabla R$ only. First note that the curvature terms only depend on $g$ and hence are bounded. Since $\int_M u\ \vol =1$ and $E(\Om(t))$ is decreasing, it follows that (\ref{curvatureterm}) is bounded by  
\[C(1+\Theta(\Om(t))),\]
    where $C$ is a constant determined by the geometry of $(M^8,g)$. 
    For the second term of (\ref{preAlmostMonotonicityEquation}), using again that $E(\Om(t))$ is decreasing, a standard argument shows that we can bound it by
\[C(E(\Om(0)))+\log\Big( \frac{B}{(t_0-t)^4}\Big)\Theta(\Om(t)).\]
    Combining the above, we have
\begin{align}
\label{bound1}
    \frac{\partial}{\partial t} \Theta(\Om(t)) \leq &-2(t_0-t) \int_M u |\mathrm{div} T-T(\nabla f)|^2\ \vol\\
    &+C(1+\log\Big( \frac{B}{(t_0-t)^4}\Big))\Theta(\Om(t)) + C(1+E(\Om(0))).\nonumber
\end{align}
    Let $\xi(t)$ be a solution of the ODE
\[\xi'(t)=1+\log\Big( \frac{B}{(t_0-t)^4}\Big).\]
    Then we can rewrite (\ref{bound1}) as
\[\frac{\partial}{\partial t}(e^{-C \xi(t)}\Theta(\Om(t)))\leq C(1+E(\Om(0)))\]
    and the first claim now follows. 
    The second claim is immediate from the explicit expression of the backwards heat kernel, see (\ref{Euclideanheatkernel}) below.
\end{proof}

The key application for the above monotonicity formula is the following $\epsilon$-regularity theorem, which we shall use to study singularities of the flow in \S\ref{sec: singularityformation}. We follow the approach employed by Grayson-Hamilton  in the context of harmonic map heat flow \cite{GraysonHamilton1996}.

\begin{theorem}[$\epsilon$-regularity]\label{epsilonregularityGraysonHamilton}
    Let $(M^8,g)$ be a compact Riemannian manifold and let $E_0$ be a positive constant. There exist constants $\epsilon, \overline{\rho}>0$ such that, for every $\rho \in (0,\overline{\rho}]$, there exist $r\in(0,\rho)$ and $C<\infty$ with the following significance.

    Suppose $\{\Om(t)\}_{t\in[0,t_0)}$ is a solution to the harmonic QK flow (\ref{harmonicQKflowequation}), inducing $g$ and satisfying $E(\Om(0))\leq E_0$. If
\[\Theta_{(p,t_0)}(\Om(t_0-\rho^2))<\epsilon,\]
    for some $p\in M^8$, then  $\Lambda_{r}(x,t):=\mathrm{min}\Big(1-r^{-1}d_g(p,x),\sqrt{1-r^{-2}(t_0-t)}\Big)$ satisfies
\[
\Lambda_r(x,t)|T(\Om(x,t))| \leq \frac{C}{r},
\quad\forall
(x,t)\in B(x_0,r)\times [t_0-r^2,t_0].
\]
\end{theorem}
\begin{proof}
    In view of the weak almost-monotonicity formula in Theorem \ref{almostmonotonicitytheorem}, the proof of Theorem \ref{epsilonregularityGraysonHamilton} is now completely analogous to those in  \cite{Dwivedi2019}*{Theorem 5.7} and \cite{Dwivedi2021}*{Theorem 5.5}, so we shall only detail its key moments.

    Suppose by contradiction that, for any sequences $\varepsilon_i,\bar{\rho}_i \to 0$, there exist $\rho_i \in (0,\bar{\rho}_i]$ such that, given any $r_i \in (0,\rho_i)$ and $C_i\to \infty$, there exist counterexamples $\{\Om_i(t)\}_{t\in[0,t_i)}$  such that
\begin{equation}
    E(\Om(0))\leq E_0
    \qandq
    \Theta_{(p_i,t_i)}(\Om(t_i-\rho_i^2))<\epsilon_i,
\end{equation}
but
\begin{gather}
\label{contracditionequation}
    r_i\Big(\max_{B(x_i,r_i)\times [t_i-r^2_i,t_i]} \Lambda_{r_i}(x,t)|T(\Om_i(x,t))|\Big) > C_i,
\end{gather}
    for some $x_i\in M$. Setting $Q_i:=|T(\Om_i(\bar{x}_i,\bar{t}_i))|$, where $(\bar{x}_i,\bar{t}_i)$ denotes the point where the maximum is attained, we can consider the  parabolic rescaled flow
\begin{gather*}
    \widetilde{\Om}_i(t) :=Q_i^4\Om_i(\bar{t}_i+Q_i^{-2}t),
\end{gather*}
    as in Corollary \ref{corollaryparabolic}, with $c=Q_i$. Using (\ref{contracditionequation}) and the definition of $Q_i$, we find that 
\[|T(\widetilde{\Om}_i(t))(\bar{x}_i,0)|=1.\] 
    Now compactness, from Theorem \ref{thm: compactness}, implies that the limit of the rescaled flow $(M,\widetilde{\Om}_i(t),\bar{x}_i)$ is the ancient solution $(\R^8,\Om_{\infty}(t),0)$ and satisfies $|T({\Om}_{\infty})(0,0)|=1.$
    On the other hand, taking the limit of $\Theta$ in the monotonicity formula of Theorem \ref{almostmonotonicitytheorem} shows that $|T({\Om}_{\infty})(0,0)|=0$, which gives the desired contradiction.
\end{proof}
Although the almost-monotonicity formula of Theorem \ref{almostmonotonicitytheorem} is sufficient for analysing singularities of the flow, we shall need a more refined monotonicity formula to obtain long time existence given small initial energy. To this end we modify the functional $\Theta$ as follows. 

We may assume, without loss of generality,  that $(M^8,g)$ has injectivity radius at least $1$, and introduce geodesic normal coordinates $x_i$ in a unit ball around any given point $p\in M$  via the exponential map 
$$\exp\big|_p:B(0,1)\subset \R^8\cong T_pM\to B(p,1)\subset M.
$$
Let $\phi$ be a test function on $\R^8$, with compact support in $B(0,1)$
and constant on $B(0,1/2)$, and let $G$ denote the usual Euclidean backward heat kernel on $\R^8$:
\begin{equation}
\label{Euclideanheatkernel}
    G=\frac{1}{(4\pi(t_0-t))^4} \exp(-\frac{|x|^2}{4(t_0-t)}),
\end{equation}
where $|x|^2=x_1^2+\cdots+x_8^2$. For $0<t<t_0\leq \tau$, we define
\[
    Z(t) = (t_0-t)\int_M |T(t)|^2G\phi^2 \vol
    = (t_0-t)\int_{\R^8}|T(t)|^2G\phi\sqrt{|g|}dx.
\]
We should emphasise that $|\cdot|$ here denotes the norm with respect to $g$ (not the Euclidean metric) and also that the integrand is only supported on $B(p,1)\cong B(0,1)$. In contrast to the functional $\Theta$ defined by
(\ref{ThetaFunctional}), observe that 
now we are using the Euclidean heat kernel $G$, rather than $u$, and we are only working locally in geodesic unit balls where the function $\phi$ is supported.
As with $\Theta(t)$, we shall compute the evolution of $Z(t)$. 

A subtle point here is that, in geodesic normal coordinates at $p$, the metric $g$ is approximately Euclidean, and we already saw in Theorem \ref{almostmonotonicitytheorem} that $\Theta(t)$ is indeed monotone on Euclidean $\R^8$, so the trick is to exploit this approximation in $B(p,1)$ using the functional $Z(t)$. This was done for the harmonic map heat flow (of maps) by Chen and Struwe \cite{ChenStruwe1989}, and it was recently adapted to the harmonic flow of almost-Hermitian structures by He and Li \cite{HeLi2021}, albeit with some subtleties, see Remark \ref{rem: compare to He-Li} below.

\begin{theorem}
\label{monotonicitytheorem1}
    For any $N>1$ and $t_1,t_2$ such that 
    $$
    t_0-\mathrm{min}\{0,t_0\}<t_1\leq t_2<t_0,
    $$
    the following monotonicity formula holds:
\[
    Z(t_2)\leq e^{C(f(t_2)-f(t_1))}Z(t_1)+C\Big(N^{4}(E(0)+\sqrt{E(0)}+\frac{1}{\log^2N})\Big)(t_2-t_1),
\]
    where $C$ is a constant depending only on $(M,g)$, and
\[f(t)=(t_0-t)
\left\{
-26+26 \log(t_0-t) 
-13\log^2(t_0-t)
+4\log^3(t_0-t)
-\log^4(t_0-t)
\right\}.\]
\end{theorem}
\begin{proof}
    First we compute
\begin{align}
    \frac{d}{dt}Z(t)= 
    &- \int_M |T|^2G\phi^2 \vol\label{diffZt} \\
    &+2(t_0-t)\int_M g(\nabla\mathrm{div}(T), T)G\phi^2 \vol\nonumber\\ &+\int_M |T|^2\Big(\frac{n}{2}-\frac{|x|^2}{4(t_0-t)}\Big)G\phi^2 \vol, \nonumber
\end{align}
    using Corollary \ref{evolutionoftorsionsquare} for the second term, and the last term comes from differentiating $G$.
    Integrating by parts, we have
\begin{align}
    \int_M g(\nabla\mathrm{div}(T), TG\phi^2) \vol = 
    &- \int_M g(\mathrm{div}(T), \mathrm{div}(T)G\phi^2 +T(\nabla G)\phi^2 +2T(\nabla \phi)G\phi) \vol
    \label{firstintegrationbyparts}\\
    = &-\int_M \Big|\mathrm{div}(T)-\frac{T(x_i\nabla x_i)}{2(t_0-t)}\Big|^2G\phi^2 \vol\nonumber\\
    &-\int_M -\Big|\frac{T(x_i\nabla x_i)}{2(t_0-t)}\Big|^2 G\phi^2 + g(\mathrm{div}(T),\frac{T(x_i\nabla x_i)}{2(t_0-t)})G\phi^2\vol\nonumber\\
    &-\int_M g(\mathrm{div}(T),2T(\nabla \phi) G\phi) \vol,\nonumber
\end{align}
    where we used that $\nabla G=-\frac{x_i\nabla x_i}{2(t_0-t)}G$, as a vector field. Combining the above, we have
\begin{align*}
    \frac{d}{dt}Z(t)= &- \int_M |T|^2G\phi^2 \vol \\
    &-2(t_0-t)\int_M \Big|\mathrm{div}(T)-\frac{T(x_i\nabla x_i)}{2(t_0-t)}\Big|^2G\phi^2 \vol\\
    &-2(t_0-t)\int_M -\Big|\frac{T(x_i\nabla x_i)}{2(t_0-t)}\Big|^2 G\phi^2 + g(\mathrm{div}(T),\frac{T(x_i\nabla x_i)}{2(t_0-t)})G\phi^2\vol\\
    &-2(t_0-t)\int_M g(\mathrm{div}(T),2T(\nabla \phi) G\phi) \vol\\ &+\int_M |T|^2\Big(\frac{n}{2}-\frac{|x|^2}{4(t_0-t)}\Big)G\phi^2 \vol\\
    =\ &I+II+III+IV+V.
\end{align*}
The goal is now to carefully estimate the right-hand side; we begin with $IV$.
\begin{align*}
    |IV| \leq\ &2(t_0-t)\int_M \Big|g\big((\mathrm{div}(T)-\frac{T(x_i\nabla x_i)}{2(t_0-t)})G^{1/2}\phi,2T(\nabla \phi)G^{1/2}\big)\Big|\vol\\
    &+ \int_M \Big|g\big(T(x_i\nabla x_i),2T(\nabla \phi)G\phi\big)\Big|\vol\\
    \leq\ &\frac{1}{2}|II|+4(t_0-t)\int_M |T|^2|\nabla\phi|^2G\vol+2\int_M |T|^2|g(x_i\nabla x_i,\nabla\phi)|G\phi\vol
\end{align*}
    where we used the triangle inequality in the first line and then Young and Cauchy-Schwartz inequalities. Now recall that $\nabla\phi=0$ in $B(0,1/2)$, and $G$ concentrates at $x=0$ as $t\to t_0$. There are now two cases to consider: if $t_0-t>1/N$, then $G< C N^{n/2}$, whereas if $t_0-t<1/N<1$ then $G |\nabla \phi|< C$, since $G$ is bounded outside $B(0,1/2)$. Hence either way we find
\[|IV|\leq \frac{1}{2}|II|+C N^{n/2}E(0).\]
Next we consider $III$.
\begin{align*}
III =\ &\frac{1}{2(t_0-t)}\int_M |T(x_i\nabla x_i)|^2 G\phi^2\vol -\int_M g(\mathrm{div}(T),T(x_i\nabla x_i))G\phi^2\vol\\
=\ &\frac{1}{2(t_0-t)}\int_M |T(x_i\nabla x_i)|^2 G\phi^2\vol +\int_M g(T,\nabla(T(x_i\nabla x_i)G\phi^2))\vol\\
=\ &\int_M g(T,\nabla(T(x_i\nabla x_i))G\phi^2)\vol+2\int_M g(T(\nabla \phi), T(x_i\nabla x_i))G\phi)\vol
\end{align*}
where we integrated by parts in the second line and again used that $\nabla G=-\frac{x_i\nabla x_i}{2(t_0-t)}G$ in the last line.
We also have
\begin{align}
   \int_M g(T,\nabla(T(x_i\nabla x_i))G\phi^2)\vol 
   &= 
   \int_M g(T,(\nabla T)(x_i\nabla x_i)G\phi^2)\vol+
   \int_M g(T,T(\nabla(x_i\nabla x_i))G\phi^2)\vol
   \label{usingbianchiidentity}\\
   &= 
   \int_M g(T,(\nabla_{x_i\nabla x_i} T+R(\cdot,x_i\nabla x_i))G\phi^2)\vol\nonumber\ +\\
   &\ \ \ \ \int_M g(T,T(\nabla(x_i\nabla x_i))G\phi^2)\vol\nonumber\\
   &=\int_M \big(\frac{1}{2} \nabla_{x_i\nabla x_i} |T|^2+g(T,R(\cdot,x_i\nabla x_i))\big)G\phi^2\vol\nonumber \ +\\
   &\ \ \ \ \int_M g(T,T(\nabla(x_i\nabla x_i))G\phi^2)\vol\nonumber
\end{align}
where we used the Bianchi-type identity (\ref{eq: bianchi identity2}) in the second line and the fact that $T\in \Om^1 \otimes \Om^2_{15}$. Note that, in local coordinates $\nabla x_i=g^{ik}\partial_{x_k}$ (using Einstein's summation convention) and hence in the above expression, \[\nabla(x_i\nabla x_i)=x_ig^{ik}\nabla {\partial_{x_k}}+g^{ik}dx_i\otimes \partial_{x_k}+x_i\partial_{x_j}(g^{ik})dx_j \otimes \partial_{x_k}\]
denotes an endomorphism.
Since $\vol=\sqrt{|g|}dx$, in local coordinates we compute
\begin{align}
    \int_M \nabla_{x_i\nabla x_i} |T|^2 G\phi^2 \vol = &\int_{\R^n} x_ig^{ik}\partial_{x_k}|T|^2G\phi^2\sqrt{|g|} dx \label{secondintegrationbyparts}\\
    = &-\int_{M}g^{ii} |T|^2G\phi^2 \vol+\int_{M}g^{ik} |T|^2(\frac{x_i x_k}{2(t_0-t)})G\phi^2 \vol\nonumber\\  &-\int_{\R^n} |T|^2x_iG\partial_{x_k}(g^{ik}\phi^2\sqrt{|g|}) dx\nonumber
\end{align}
where we now integrated by parts on $\R^n$. Combining all of the above, we have so far
\begin{align}
    I+III+V 
    =&\int_M g(T,g^{ik}dx_i\otimes T( \partial_{x_k})-T)G\phi^2\vol\nonumber\\
    &-\frac{1}{4(t_0-t)}\int_M |T|^2\Big({|x|^2}-{g^{ik}x_i x_k}\Big)G\phi^2 \vol\label{quarticterminx}\\
    &+\frac{1}{2}\int_{M}({n-g^{ii}}) |T|^2G\phi^2 \vol 
    +\int_M g(T,R(\cdot,x_i\nabla x_i))G\phi^2\vol\nonumber\\
    &+\int_M g(T,T(x_ig^{ik}\nabla {\partial_{x_k}}+x_i\partial_{x_j}(g^{ik})dx_j \otimes \partial_{x_k})G\phi^2)\vol\nonumber\\
    &-\frac{1}{2}\int_{\R^n} |T|^2x_iG\partial_{x_k}(g^{ik}\phi^2\sqrt{|g|}) dx 
    +2\int_M g(T(\nabla \phi), T(x_i\nabla x_i))G\phi)\vol.\nonumber
\end{align}
Since we are working in $B(p,1)$ with geodesic normal coordinates we know that $g^{ij}=\delta_{ij}+O(|x|^2)$ and $\Gamma_{ij}^k=O(|x|)$ in a neighbourhood of $x=0$, and hence, using
\[\big|\int_M g(T,R(\cdot,x_i\nabla x_i))G\phi^2\vol\big|\leq C\int_M|T||x|G\phi^2\vol,\]
we get the bound
\begin{equation}
    |I+III+V|\leq CE(0)+\frac{C}{(t_0-t)}\int_M |T|^2|x|^4 G\phi^2\vol+{C}\int_M |T|^2|x|^2 G\phi^2\vol+C\int_M|T||x|G\phi^2\vol.\nonumber
\end{equation}
Note that here we again used the fact that $G$ can only concentrate at $x=0$, but $|\nabla \phi|$ and the $\Gamma_{ij}^k$ vanish at $x=0$. 
We now have the estimate
\begin{align}
    \frac{1}{(t_0-t)}\int_M |T|^2|x|^4 G\phi^2\vol =\ &\frac{1}{(t_0-t)}\int_{|x|^2\leq (t_0-t)\log^2(t_0-t)} |T|^2|x|^4 G\phi^2\vol\nonumber\\
    &+ \frac{1}{(t_0-t)}\int_{|x|^2> (t_0-t)\log^2(t_0-t)} |T|^2|x|^4 G\phi^2\vol\nonumber\\
    \leq\ &{\log^4(t_0-t)}Z(t)+ \frac{\exp\big(-\log^2(t_0-t)/4\big)}{(t_0-t)^{n/2+1}}C E(0)\nonumber\\
    \leq\ &{\log^4(t_0-t)}Z(t)+ C E(0),\nonumber
\end{align}
where we used that $\frac{\exp\big(-\log^2(t_0-t)/4\big)}{(t_0-t)^{n/2+1}}$ is uniformly bounded, for $0<t<t_0$. Analogously, we get 
\begin{equation*}
   {C}\int_M |T|^2|x|^2 G\phi^2\vol
    \leq {C}{\log^2(t_0-t)}Z(t)+ C E(0).
\end{equation*}
As argued above, for $t_0-t>1/N$ we have $G<CN^{n/2}$, so
\[\int_M|T||x|G\phi^2\vol \leq CN^{n/2}\sqrt{E(0)}.\]
On the other hand, if $t_0-t<1/N$, then Young's inequality gives
\begin{align*}
    \int_M|T||x|G\phi^2\vol &\leq  \frac{1}{4 \log^2(t_0-t)}\int_M \frac{|x|^2}{t_0-t}G\phi^2 \vol+\log^2(t_0-t)Z(t)\\
    &\leq \frac{C}{\log^2 N} +\log^2(t_0-t)Z(t).
\end{align*}
    Gathering all terms,
\[\frac{d}{dt}Z(t)\leq -\frac{1}{2}|II|+C(\log^4(t_0-t) +\log^2(t_0-t)) Z(t)+C N^{n/2}(E(0)+\sqrt{E(0)})+\frac{C}{\log^2 N}.\]
    Now, the function $f(t)$ satisfies $f'(t)=\log^4(t_0-t) +\log^2(t_0-t)$, so 
\begin{align*}
    \frac{d}{dt}\Big(e^{-Cf}Z(t)\Big) &= e^{-Cf}\Big(\frac{d}{dt}Z(t)-C(\log^4(t_0-t) +\log^2(t_0-t))\Big) \\
    &\leq C e^{-Cf} \Big( N^{n/2}(E(0)+\sqrt{E(0)})+\frac{1}{\log^2 N}\Big).
    \qedhere
\end{align*}
\end{proof}
\begin{remark}
\label{rem: compare to He-Li}
    A similar approach can be found in \cite{HeLi2021}*{Theorem 3.1}, in the case of almost Hermitian structures, with $H=\rU(n/2)$. We must highlight however two important differences, stemming from what we believe to be a minor overlook of some features in the original proof by Chen-Struwe  \cite{ChenStruwe1989} for the harmonic heat flow of \emph{maps}, as opposed to \emph{tensors}, eventually leading us to a different function $f$ in the monotonicity formula. 
    
    First, the authors integrate by parts in (\ref{firstintegrationbyparts}) with respect to the Euclidean metric, rather than $g$, and this results in an additional term involving derivatives of $g$ (confusingly denoted by $\nabla g_{ij}$ therein). This is indeed the procedure adopted in \cite{ChenStruwe1989}*{Lemma 4.2}, where it is not problematic because there are no covariant derivatives of tensors involved, only partial derivatives of maps. 
    
    Second, our integration by parts in (\ref{secondintegrationbyparts}) gives rise to a term involving $g^{ik}x_ix_j$, which is equal to $|x|^2$ only to zeroth order in the unit geodesic ball, since $g^{ij}=\delta_{ij}+O(|x|^2)$. This yields a term of order $|x|^4$ in (\ref{quarticterminx}), which does not appear in \cite{HeLi2021}. It is the bound on this term that finally requires a different choice of $f$.
\end{remark}

\begin{remark}
\label{remarkongeneralcase}
    We now highlight the key features of the proof that generalise immediately to other structure groups $H\subset \SO(n)$. First,  we need the squared norm of $T\in\Om^1 \otimes \mathfrak{m}\subset \Om^1 \times \Om^2$ to evolve by 
$$\frac{\partial}{\partial t}|T|^2=g(\nabla \mathrm{div}(T),T)$$
    which is used in (\ref{diffZt}).
    Second, we need the Bianchi identity to establish that $$\pi_{\mathfrak{m}}((\nabla_X T)(Y)-(\nabla_Y T)(X))=\pi_{\mathfrak{m}}(R(X,Y))$$
which is used in (\ref{usingbianchiidentity}). Aside from these two ingredients, the rest of the calculations is completely independent of the structure group $H$. From the results in \cites{Dwivedi2019, Loubeau2019} and \cite{Dwivedi2021} we immediately deduce that Theorem \ref{monotonicitytheorem1} also applies to the cases of $H=G_2$ and $H=\Spin(7)$.
\end{remark}

Next we define the functional
\begin{equation}
    \Psi(R) = \int_{t_0-4R^2}^{t_0-R^2}\int_M |T(t)|^2G\phi^2 \vol
    = \int_{t_0-4R^2}^{t_0-R^2}\int_{\R^n} |T(t)|^2G\phi^2\sqrt{|g|} dx.
\end{equation}
and a similar argument as in \cite{HeLi2021}*{Theorem 3.2} now yields:
\begin{theorem}
\label{monotonicitytheorem2}
    For any $N>1$ and $R_1,R_2$ such that 
    $$
    0<R_2\leq R_1<\mathrm{min}\{\sqrt{t_0}/2,1\},
    $$
    the following monotonicity formula holds:
\begin{equation}
    \Psi(R_2)
    \leq C_0e^{C(\tilde{f}(t_2)-\tilde{f}(t_1))}\Psi(R_1) +C\Big(N^{4}(E(0)+\sqrt{E(0)} +\frac{1}{\log^2N})\Big)(R_1-R_2),
\end{equation}
    where $C_0,C>0$ are constants depending only on $(M,g)$ and
\[\tilde{f}(R)
=R^2(-26+52R^2\log R-52\log^2R+64\log^3R-32\log^4R).\]
\end{theorem}
\begin{proof}
    We set $\al=R^2_2/R^2_1\leq 1$ and $\tilde{t}\leq t=\al \tilde{t}+(1-\al)t_0$. Using Theorem \ref{monotonicitytheorem1}, we compute
\begin{align*}
    \Psi(R_2) &= \int^{t_0-R^2_2}_{t_0-R^2_2} \frac{Z(t)}{t_0-t} dt = \int^{t_0-R^2_1}_{t_0-R^2_1} \frac{Z(t)}{t_0-\tilde{t}} d\tilde{t} \\
    &\leq \int^{t_0-R^2_1}_{t_0-R^2_1} e^{C(f(t)-f(\tilde{t}))}\frac{Z(t)}{t_0-\tilde{t}} +C\Big(N^{4}(E(0)+\sqrt{E(0)}+\frac{1}{\log^2N})\Big)\frac{t-\tilde{t}}{t_0-\tilde{t}} d\tilde{t} \\
    &\leq C_0 e^{C(\tilde{f}(R_2)-\tilde{f}(R_1))}\Psi(R_1)+C\Big(N^{4}(E(0)+\sqrt{E(0)}+\frac{1}{\log^2N})\Big)(R_1-R_2)
\end{align*}
In the last line we used the fact the function $f(\alpha\tilde{t}+(1-\alpha)t_0)-f(\tilde{t})$ is bounded above and decreasing as $\tilde{t}\to t_0$.
\end{proof}
The above monotonicity result is crucial in the proof of long-time existence of the harmonic QK flow. As with the almost-monotonicity formula, the argument in \cite{HeLi2021} can also be readily extended to our setup when $H=\Sp(2)\Sp(1)$, leading to long-time existence given small initial torsion. Fortunately, we need not carry out that analysis in detail, since it follows from the  theory of harmonic flows of  $H$-structures, cf. \cite{Fadel2022}. In view of Remark \ref{remarkongeneralcase}, such behaviour was indeed to be expected, even if some passages might have to be modified in comparison to \cite{HeLi2021}, due to certain subtle estimates involving the parabolic cylinder cf. \cite{Fadel2022}*{Theorem 2.10}. In any event, specialising the general theory to our present situation, we have:
\begin{theorem}
    Let $(M^8,g,\Om_0)$ denote a Riemannian manifold with an $\Sp(2)\Sp(1)$-structure such that $|\nabla \Om_0|<C$, for some constant $C>0$. Then there exists $\varepsilon>0$, depending only on $C$ and $g$, such that, if $E(\Om_0)<\varepsilon$, then the harmonic QK flow  (\ref{harmonicQKflowequation}) exists for all time and converges smoothly to a torsion-free QK structure.
\end{theorem}

We conclude this section with the following convexity result for the energy functional.
This can in fact be used to give a direct proof that sufficiently small $|T(\Om_0)|$ guarantees  long-time existence and convergence to a harmonic QK structure, cf. \cite{Dwivedi2019}*{Theorem 5.13} and \cite{Dwivedi2021}*{Theorem 5.9}. However, the recent results in \cite{Fadel2022} supersede this hypothesis, by only requiring that $E(\Om_0)$ be sufficiently small. In any event, such convexity result will likely be useful in future studies of the stability profile of  critical points.

\begin{proposition}
\label{convexitylemma}
    Along a solution of the QK harmonic flow (\ref{harmonicQKflowequation}) on $(M^8,g)$, we have
\begin{equation}
    \frac{d^2}{dt^2} E(\Om(t)) \geq \int_M (\Lambda -|T|^2)|\mathrm{div}\ T|^2\ \vol,
\end{equation}
    where $\Lambda$ denotes the first non-zero eigenvalue of the rough Laplacian of $g$ on $2$-forms.
\end{proposition}
\begin{proof}
    From (\ref{evolutionoftorsionsquareequation}), we have
\begin{equation*}
    \frac{d^2}{dt^2} E(\Om(t)) = -2 \int_M g(\frac{d}{dt}\mathrm{div}\ T,\mathrm{div}\ T)\ \vol.
\end{equation*}
    Since the divergence operator only depends on the metric it commutes with $\frac{d}{dt}$, hence integrating by parts and using (\ref{evolutionoftorsionequation}) we have 
\begin{align*}
    \frac{d^2}{dt^2} E(\Om(t))
    =\ &2 \int_M |\nabla(\mathrm{div}\ T)|^2\ \vol\ +\\ &\frac{1}{16}\int_M g(
    (T\diamond \Om ) \ip_3
    (\mathrm{div}\ T \diamond \Om)-(\mathrm{div}\ T \diamond \Om)\ip_3 (T\diamond \Om ),\nabla(\mathrm{div}\ T))\ \vol.
\end{align*}
Using (\ref{Lm215andLm215pairing}) together with Young's inequality,
\begin{align*}
    \int_M g((T\diamond \Om ) \ip_3
    (\mathrm{div}\ T \diamond \Om),\nabla(\mathrm{div}\ T))\ \vol &=
    16\sum_{i,j,k,m}\int_M T_{m;ik}(\mathrm{div}\ T)_{jk} \nabla_m(\mathrm{div}\ T)_{ij}\ \vol\\
    &\geq -8 \int_M |T|^2|\mathrm{div}\ T|^2+|\nabla(\mathrm{div}\ T)|^2\ \vol.
\end{align*}
    The same argument applies to the term involving $(\mathrm{div}\ T \diamond \Om)\ip_3 (T\diamond \Om )=- (T \diamond \Om)\ip_3 (\mathrm{div}\ T\diamond \Om )$. So combining the above we have 
\begin{align*}
    \frac{d^2}{dt^2} E(\Om(t))
    &\geq\ 2 \int_M |\nabla(\mathrm{div}\ T)|^2\ \vol-\int_M |\mathrm{div}\ T|^2|T|^2+|\nabla(\mathrm{div}\ T)|^2\ \vol\\
    &=\int_M |\nabla(\mathrm{div}\ T)|^2 - |\mathrm{div}\ T|^2|T|^2\ \vol\\
    &\geq \int_M (\Lambda -|T|^2)|\mathrm{div}\ T|^2\ \vol.
\end{align*}
    On a compact manifold, the kernel of the rough Laplacian consists of parallel $2$-forms, so indeed it is orthogonal to $\mathrm{div}\ T$.
\end{proof}
Note that Lemma \ref{convexitylemma} does differ depending on the $H$-structure. For instance, in the $\Spin(7)$-case we have
\begin{equation*}
    \frac{d^2}{dt^2} E(\Phi(t)) \geq \int_M (\Lambda -3|T|^2)|\mathrm{div}\ T|^2\ \vol,
\end{equation*}
see \cite{Dwivedi2021}*{Lemma 5.7}. We refer the reader to \cite{Dwivedi2019}*{Lemma 5.11} for the $\rG_2$ case.

\subsection{Singularities of the flow}
\label{sec: singularityformation}

In this section we investigate the formation of singularities along the harmonic QK flow. Let us consider a solution $\{\Om(t)\}$ to (\ref{harmonicQKflowequation}) defined for $t\in [0,T_{max})$, and define the \emph{singular set} $S$ of the flow by
\begin{equation}
\label{eq: def sing S}
    S=\{x\in M^8\ |\ \Theta_{x,\tau}(\Om(\tau-\rho^2))\geq \epsilon,
    \;\forall 
    \rho \in [0,\overline{\rho})\},
\end{equation}
where $\epsilon$ and $\overline{\rho}$ are as in Theorem \ref{epsilonregularityGraysonHamilton}. This wording is justified by the next result, in the same vein as   \cite{GraysonHamilton1996}*{Theorem 4.3}.
\begin{theorem}
    Let $\{\Om(t)\}_{t\in[0,T_{max})}$ denote the maximal smooth solution to (\ref{harmonicQKflowequation}) starting at $\Om(0)=\Om_0$ with $T_{max} < \infty$. As $t\to T_{max}$, $\Om(t)$ converges smoothly to a QK $4$-form $\Om_{T_{max}}$, away from the closed set $S$. Moreover, $S$ has finite $6$-dimensional Hausdorff measure and satisfies
\[\mathcal{H}^6(S)\leq C E_0\]
    where $C\geq0$ is a  constant depending only on $g$.
\end{theorem}
\begin{proof}
    We adapt a similar scheme of proof as in \cite{GraysonHamilton1996}*{Theorem 4.3}, also found in \cite{Dwivedi2021}*{Theorem D}.

    First let's assume that $\mathcal{H}^6(S)$ is finite.
    Since $S$ is a closed set of finite $6$-dimensional measure, Theorem 4.2 of  \cite{GraysonHamilton1996} asserts that there exists $S'\subset S$ such that 
\[\mathcal{H}^6(S')\geq \frac{1}{2}\mathcal{H}^6(S).\]
    Now, the solution to the backwards heat equation is given by
\[
    u_{S'}(x,t)=\int_{y\in S'} u_{y,T_{max}}(x,t) d\mathcal{H}^6(y)
\]
    and it satisfies
\begin{equation}
\label{eq: bound on uS'}
    u_{S'}(x,t) 
    \leq \frac{C}{T_{max} -t}.
\end{equation}
    From the definition of $S$ in \eqref{eq: def sing S}, we have that
\[\epsilon \mathcal{H}^6(S')
=\int_{S'} \epsilon d\mathcal{H}^6(y)
\leq \int_{S'}\Theta_{(y,T_{max})}(\Om(T_{max}-\rho^2)) d\mathcal{H}^6(y),
\]
    so using definition \ref{ThetaFunctional} of $\Theta$
    and (\ref{eq: bound on uS'}) we have
\begin{equation*}
    \epsilon \mathcal{H}^6(S') 
    \leq \int_{S'}\int_M \rho^2 u_{(y,T_{max})}(x,T_{max}-\rho^2) |T(T_{max} - \rho^2)|^2d\mathcal{H}^6(y) \leq C E_0.
\end{equation*}
    To conclude the proof note that if instead $\mathcal{H}^6(S)$ was infinite then one could choose subset $S'\subset S$ with arbitrarily large $6$-dimensional Hausdorff measure, but repeating the above argument would give a contraction. 
    So we must have that indeed $\mathcal{H}^6(S)< \infty$ and this yields the result.
\end{proof}

Note that, for each $x\in S$, one can find a sequence $(x_i,t_i) \to (x,T_{max})$ such that $$\lim_{i \to \infty} |T(\Om(x_i,t_t))| \to \infty$$
cf. \cite{Loubeau2019}*{Theorem 2}, so indeed $S$ is the singular set of the flow.

\begin{remark}
    In the context of quaternion-K\"ahler geometry, the only distinguished classes of submanifolds are \emph{quaternionic submanifolds} i.e. submanifolds calibrated by $\Om$. 
    So it is natural to expect that singularities for the QK flow (\ref{harmonicQKflowequation}) would occur along such submanifolds. The result of Dadok et al. in \cite{Dadok1988}*{Theorem 3.7} and Harvey-Lawson in \cite{Harvey1982}*{Section V} assert that the only possible quaternionic submanifolds in $\mathbb{H}^{2}$ and $\mathbb{H}\mathbb{P}^{2}$ are $\mathbb{H}$ or $\mathbb{H}\mathbb{P}^1=S^4$;  this is significantly more restrictive than in the K\"ahler setting, in which plenty of examples can be easily generated by polynomials in $\C^n$ and $\mathbb{C}\mathbb{P}^n$. 
\end{remark}

We shall now show that \textit{type-I singularities} of the harmonic  QK flow are in fact modelled on shrinking solitons. Based on the analogy with the harmonic map flow, motivated from the results in  \cite{GraysonHamilton1996}
we define type-I singularities as follows: 
    
\begin{definition}
    A solution $\{\Om(t)\}_{t\in[0,T_{max})}$ to the harmonic QK flow (\ref{harmonicQKflowequation}) is said to encounter
    a \emph{type-I singularity} at $T_{max}$ if
\[\sup_{x\in M}|T(t)|\leq \frac{1}{\sqrt{C(T_{max}-t)}},\]
    where $C>0$ is a constant. 
    If a singularity does not satisfy the above bound, then it is said to be a \emph{type-II singularity}.
\end{definition}
\begin{theorem}
\label{type1singularitymodel}
    Suppose that a solution $\{\Om(t)\}_{t\in[0,T_{max})}$ to the harmonic flow (\ref{harmonicQKflowequation}) encounters a type-I singularity at $T_{max}$. Let $x \in M$ and $\lambda_i\searrow 0$ be a decreasing sequence, and consider the parabolic rescaled solution 
    $$
    \Om_i(t):=\lambda_i^{-4}\Om(T_{max}+\lambda^2_it).
    $$ 
    Then $(M,g_{\Om_i}, \Om_i(t),x)$ subconverges smoothly to an ancient solution $(\R^8,g_{0}, \{\Om_{\infty}(t)\}_{t<0},0)$ induced by a  shrinking soliton,  i.e.
\begin{equation}
    \mathrm{div}( T_{\Om_\infty})(x,t) = -T_{\Om_\infty}\Big(\frac{x}{2t}\Big).
\end{equation}
    Moreover, $x\in M\backslash S$ if, and only if, $\Om_{\infty}(t)$ is the stationary flow induced by a torsion-free QK structure on Euclidean $\R^8$.
\end{theorem}
\begin{proof}
    From Corollary  \ref{corollaryparabolic}, we see that $\{\Om_i(t)\})$ is well-defined for $t\in [-\lambda_i^{-2}T_{max},0)$.
    The fact that the limit is an ancient solution on $\R^8$ now follows from the compactness in Theorem \ref{thm: compactness}. We  conclude that the limit is a shrinking soliton by Theorem \ref{almostmonotonicitytheorem}, see also \cite{GraysonHamilton1996}*{Theorem 5.1}.
\end{proof}

It was recently shown that finite-time singularities do occur for the harmonic flow of almost Hermitian structures \cite{HeLi2021}. Their construction can be easily adapted to our setting, see also \cite{Fadel2022}*{Theorem 2.16}. 
However, it is worth mentioning that those examples are based on a non-constructive argument, and as such the concrete nature of the singularity is unknown. Next we shall construct several \emph{explicit} solutions to the harmonic flow illustrating long-time existence and convergence to both torsion free and (non-trivial) harmonic $\Sp(2)\Sp(1)$-structures.

\section{Explicit solutions of the harmonic flow}
\label{sec: explicit sols of the HF}

In this section we construct explicit solutions to the harmonic QK flow (\ref{harmonicQKflowequation}). In particular, we exhibit convergence to harmonic QK structures in infinite time; examples of a similar flavour for the harmonic $\rG_2$ flow on $S^7$ were found in \cite{Loubeau2022}. We also construct a steady harmonic $\Sp(2)\Sp(1)$-soliton (see Definition \ref{definitionofsoliton}), which to the best of our knowledge is in fact the first nontrivial concrete example of a harmonic soliton for any $H$-structure. 

\subsection{Eternal solutions}
\label{sec: explicitflowsolutionsection}

We find two eternal solutions to the harmonic QK flow (\ref{harmonicQKflowequation}). 
The first example we describe is on $\R^8$, endowed with the quaternionic hyperbolic metric as in sub-section \ref{sec: example2}.
We modify the torsion-free QK $4$-form in a suitable way to a non torsion-free one then show that the harmonic flow indeed converges back to the torsion-free solution in infinite time. 
The second example we describe is on $\SU(3)$, which we endow with a left-invariant metric. From the results of Poon and Salamon in \cite{PoonSalamon1991}, we know that there are no torsion-free QK structures on $\SU(3)$, so it is especially interesting to understand harmonic structures in this situation as the next `best' possible QK structures. We show that the flow in this case converges to a left-invariant harmonic QK structure and that moreover it induces a hypercomplex structure on $\SU(3)$ cf. \cite{Joyce1992}*{Section 3, Example 1}. The latter is a new example of a strictly harmonic QK structure.

We begin by proving the following elementary result:
\begin{proposition}
    Suppose that $\Om_0$ is invariant under an isometry $f:M\to M$ and that $\Om(t)$ is the solution to the harmonic flow (\ref{harmonicQKflowequation}). Then $\Om(t)$ is also invariant under $f$.
\end{proposition}
\begin{proof}
    Given $\Om(t)$, we can define another solution $\Psi(t)=f^*\Om(t)$ to (\ref{harmonicQKflowequation}) since
\[\frac{\partial}{\partial t}(\Psi(t))=f^*\frac{\partial}{\partial t}(\Om(t))=f^*(\mathrm{div}T \diamond \Om)=\mathrm{div}(f^*T) \diamond \Psi,\]
    where we used the facts that $\diamond$ and $\nabla$ only depend on the metric $g$, and that $f$ is an isometry. Since $f^*\Om(0)=\Om(0)$, by uniqueness of the flow cf. \cite{Loubeau2019}*{Theorem 1} it follows that $f^*\Om(t)=\Om(t)$.
\end{proof}

As a consequence, if $\Om_0$ is $G$-invariant then so is $\Om(t)$. In particular, for invariant $\Om_0$ on a homogeneous space, the harmonic flow reduces to an ODE system in $t$, so this provides a natural set up to study long-time behaviour and (possible) finite-time singularities. Our examples shall exhibit the former behaviour.

\subsubsection{Convergence to a torsion-free solution on \texorpdfstring{$\mathcal{H}\mathbb{H}^2$}{}}

Having shown in Section \ref{examplesofharmonicQKsection} that there indeed exist non-torsion-free QK structures with divergence-free intrinsic torsion, let us describe an explicit eternal solution to the flow converging (in infinite time) to a torsion-free solution. 

As in \S\ref{sec: example2}, we shall again consider the hyperbolic quaternionic plane $\mathcal{H}\mathbb{H}^2$, but now viewed as a solvable Lie group with the coframe $\{E^i\}$ satisfying the following structure equations:
\[\begin{array}{ccl}
   dE^1 = -E^{18}, &  & dE^5 = -2E^{13}+2E^{24}-2E^{58}, \\
   dE^2 = -E^{28}, &  & dE^6 = -2E^{14}-2E^{23}-2E^{68}, \\
   dE^3 = -E^{38}, &  & dE^7 = +2E^{12}+2E^{34}-2E^{78}, \\
   dE^4 = -E^{48}, &  & dE^8 = 0.
\end{array}\]
The equivalence with the cohomogeneity one description of $\mathcal{H}\mathbb{H}^2$ in \S\ref{sec: example2} can easily be seen by setting $E^1=(1-s)^{-1/4}dx_1$, $E^5=(1-s)^{-1/2}\al_1$ and so on.

Consider now the $1$-parameter family of QK $4$-forms $\Om_{a,b}$ defined by (\ref{qk4form}), where we take
\begin{align*}
    \om_1 &= E^{12}+E^{34}+E^{56}+E^{87},\\
	\om_2 &= E^{13}+E^{42}+(b E^5 +a E^6) \w E^{8}+E^{7} \w (-a E^5+b E^6),\qwithq a^2+b^2=1,
	\\
	\om_3 &= E^{14}+E^{23}+(b E^5 +a E^6)\w E^{7}+(-a E^5+b E^6) \w E^{8}.
\end{align*}
	Geometrically, we are rotating the $1$-forms $E^5$ and $E^6$ by $(a,b)\in U(1)$, which ensures that $\Om_{a,b}$ induces the same metric. The solution $(a,b)=(0,1)$ corresponds to the unique torsion-free QK $4$-form in this $1$-parameter family. A long but straightforward computation shows that 
\[\nabla_{\nabla_{E_i}E_i} \Om =0\]
and
\begin{align*}
    T(E_1)
    &=16a(-E^{17}+E^{28}+E^{36}-E^{45})+16(1-b)(E^{18}+E^{27}-E^{35}-E^{46}),\\
    T(E_2)
    &=16a(-E^{18}-E^{27}-E^{35}-E^{46})+16(1-b)(-E^{17}+E^{28}-E^{36}+E^{45}),\\
    T(E_3)
    &=16a(-E^{16}+E^{25}-E^{37}+E^{48})+16(1-b)(E^{15}+E^{26}+E^{38}+E^{48}),\\
    T(E_4)
    &=16a(E^{15}+E^{26}-E^{47}-E^{38})+16(1-b)(E^{16}-E^{25}-E^{37}+E^{48}),\\
    T(E_5)
    &=16a(E^{14}+E^{23}-E^{57}-E^{68})+16(1-b)(E^{13}-E^{24}+E^{58}-E^{67}),\\
    T(E_6)
    &=16a(-E^{13}+E^{23}+E^{58}-E^{67})+16(1-b)(E^{14}+E^{23}+E^{57}+E^{68}),\\
    T(E_7)&=
    T(E_8)=0.
\end{align*}
From this one finds that
\begin{equation}
    \mathrm{div}\ T = -192 a (E^{12}+E^{34}-E^{56}-E^{87}).
\end{equation}
Using the definition of the infinitesimal diamond action (\ref{diamondoperator}), we can compute the harmonic QK flow (\ref{harmonicQKflowequation}) and it turns out that the flow preserves the above ansatz (which is in fact what motivated this choice in the first place). Due to the symmetry of the problem, the latter reduces to the single ODE:
\[\frac{d}{dt}a(t) = -768  a(t) \sqrt{1-a(t)^2}.\]
One easily solves the latter to find the eternal solution
$$
a(t)=\frac{1}{\cosh{(768 t)}}
\qandq b(t)={\tanh{(768 t)}},
\quad\forall t \in \R.
$$ 
Moreover,  $\displaystyle\lim_{t \to \infty}(a(t),b(t)) = (0,1)$, i.e. we indeed converge to the torsion-free QK structure of the hyperbolic quaternionic plane.

From the above, the total Dirichlet energy would formally be given by
\[E(\Om(t))=6144(1-b(t))\int_M \vol_M,\]
which is infinite (unless $b\neq 1$), since $M$ has infinite volume. 
However, the normalised quantity 
\[\lim_{r\to\infty}\frac{\int_{B(0,r)}|T|^2\vol_M}{\int_{B(0,r)}\vol_M}=6144(1-b(t))\]
is well-defined and indeed converges to 0 as $t\to \infty$. The above example shows that, even if the harmonic QK flow is not rigorously speaking the gradient flow of the Dirichlet energy functional on this non-compact manifold, it still exhibits some of its informally expected properties. 
Next we exhibit a compact example.

\subsubsection{Convergence to a harmonic solution on \texorpdfstring{$\SU(3)$}{} which is not torsion-free}
\label{solutiononsu3section}

We shall now take $M^8=\SU(3)$ and construct an $SU(3)$-invariant solution to (\ref{harmonicQKflowequation}). We first begin by expressing the Maurer-Cartan form of $\SU(3)$ explicitly as
\[\begin{pmatrix}
	i(\theta_1+\theta_2) & i\theta_3-\theta_4 & \theta_5+i\theta_6 \\
	i\theta_3+\theta_4  & i(\theta_1-\theta_2) & i\theta_7+\theta_8 \\
	-\theta_5+i\theta_6  & i\theta_7-\theta_8 & -2i\theta_1 
\end{pmatrix}\!,\]
where $\theta_i$ denote a left-invariant coframing.
We define a left invariant metric by
\[g=\theta_1^2+\theta_2^2+\theta_3^2+\theta_4^2+\theta_5^2+\theta_6^2+\theta_7^2+\theta_8^2\]
and a compatible $\Sp(2)$-structure determined by the triple
\begin{align*}
	\om_1 &= \theta_{12}+\theta_{34}+\theta_{56}+\theta_{78},\\
	\om_2 &=\theta_{1} \w (\cos(f)\theta_3+\sin(f)\theta_4)-\theta_{2}\w (-\sin(f)\theta_3+\cos(f)\theta_4)+\theta_{57}-\theta_{68},\\
	\om_3 &=\theta_{1}\w (-\sin(f)\theta_3+\cos(f)\theta_4)+\theta_{2}\w (\cos(f)\theta_3+\sin(f)\theta_4)+\theta_{58}+\theta_{67},
\end{align*}
where $f\in[0,2\pi)$ is an arbitrary constant. With $\om_i$ as above we can now define a QK $4$-form $\Om$ on $\SU(3)$ by  (\ref{qk4form}), which is indeed compatible with $g$. 
Note that $f$ can in fact be taken to be any function on $\SU(3)$ and the above will still hold, but since we are only interested in $\SU(3)$-invariant structures we shall restrict to the situation when $f$ is constant. Moreover, we also note the following special case:
\begin{proposition}
    When $f=0$, the quaternionic structure defined by the $4$-form $\Om$ is in fact a hypercomplex structure. 
\end{proposition}
\begin{proof}
    First observe that the complex $2$-form $\om_2+i\om_3$ is of type $(2,0)$, with respect to the almost complex structure $I_1$ determined by $\om_1$ and $g$. When $f=0$, a straightforward computation shows that
\begin{gather}
    d((\om_2+i\om_3)\w (\om_2+i\om_3))=2(\theta_1-i\theta_2)\w (\om_2+i\om_3)\w (\om_2+i\om_3),
    \label{hypercomplex1}\\
    d((\om_3+i\om_1)\w (\om_3+i\om_1))=2(\theta_1-i\theta_3)\w (\om_3+i\om_1)\w (\om_3+i\om_1),
    \label{hypercomplex2}\\
    d((\om_1+i\om_2)\w (\om_1+i\om_2))=2(\theta_1-i\theta_4)\w (\om_1+i\om_2)\w (\om_1+i\om_2).
    \label{hypercomplex3}
\end{gather} 
From (\ref{hypercomplex1}) we see that $d \Lm^{4,0}_{I_1} \subset \Lm^{4,1}_{I_1}$, and thus $I_1$ is in fact a complex structure. The same argument applies to $I_2$ and $I_3$, and hence $I_1,I_2,I_3$ determine a hypercomplex structure. 
\end{proof}

\begin{remark}
When $f \neq 0$, the almost complex structures $I_2$ and $I_3$ are non-integrable. For instance,  a simple computation shows that 
\[
 d((\om_3+i\om_1)\w (\om_3+i\om_1)) \w \om_2 =-12\sin(f)(\theta_{1235678}+i\theta_{1345678})
+12(\cos(f)-1)(\theta_{1245678}+i\theta_{345678}),
\]
and hence this implies that $J_2$ is non-integrable.
\end{remark}
\begin{multicols}{2}
In what follows we denote by $\{E_i\}$ the dual frame to $\{\theta_i\}$. One can view $\SU(3)$ as an $\SU(2)$-bundle over $S^5$, where the $S^3$ fibres are generated by the left invariant  vector fields $E_2,E_3,E_4$. 
Furthermore, the Hopf fibration exhibits $S^5$ as an $\mathrm{U}(1)$-bundle over $\C\mathbb{P}^2$, where the $S^1$ fibres correspond to the orbit of the vector field $E_1$. This illustrates the diagonal embedding of $\mathrm{U}(2)=\mathrm{U}(1)\SU(2)$ in $\SU(3)$:
\begin{center}
\begin{tikzcd}
	\SU(2) \arrow[hookrightarrow]{r} 
	& \SU(3)  \arrow[d]  \\
	\rU(1) \arrow[hookrightarrow]{r}
	& S^5 \arrow[d]\\
	& \C\mathbb{P}^2
\end{tikzcd}
\end{center}
\end{multicols}
In view of the above, we can interpret the $\mathrm{U}(1)$ rotation defined by $f$ as lying in the $\SU(2)$ fibre. It is worth pointing out that the metric $g$ is in fact $\SU(3)\times \mathrm{U}(2)$-invariant, where the $\mathrm{U}(2)$ corresponds to the \textit{right} action generated by $E_1,E_2,E_3,E_4$. Note that $g$ is not the bi-invariant metric of $\SU(3)$, the bi-invariant Einstein metric $g_{E}$ is instead given by
\[g_E=3\theta_1^2+\theta_2^2+\theta_3^2+\theta_4^2+\theta_5^2+\theta_6^2+\theta_7^2+\theta_8^2.\]
\begin{proposition}\label{prop: rightinvariance}
    The left $\SU(3)$-invariant $4$-form $\Om$ is also invariant under the right action of $\mathrm{U}(2)$ when $f=0$ but only $\mathrm{U}(1)^2$ invariant when $f \neq 0.$
\end{proposition}
\begin{proof}
It suffices to verify that $\mathcal{L}_{E_i}\Om=0$ for $i=1,..,4$ when $f=0$ but if $f\neq 0$ then the latter only holds for $i=1,2$.
\end{proof}

We shall now compute the intrinsic torsion $T$ of $\Om$. A long but straightforward calculation show that $T$ is given by
\begin{align*}
    T(E_1)=&\ 
    T(E_2)=0\\
    T(E_3)=&\ 8(\cos(f)-1) (\theta^{13}+\theta^{42}+\theta^{57}+\theta^{86})+8\sin(f)(\theta^{14}+\theta^{23}+\theta^{85}+\theta^{76})\\
    T(E_4)=&\ 8(\cos(f)-1) (\theta^{14}+\theta^{23}+\theta^{58}+\theta^{67})-8\sin(f)(\theta^{13}-\theta^{24}-\theta^{57}+\theta^{68})\\
    T(E_5)=&\ -(8\cos(f)+4)\theta^{15}+(8\sin(f)+12)\theta^{16}-(8\sin(f)-4)\theta^{25}+(12-8\cos(f))\theta^{26}\\&
    -(4\cos(f)-4\sin(f)-8)(\theta^{37}+\theta^{48})+4(\cos(f)+\sin(f))(\theta^{38}-\theta^{47}) \\
    T(E_6)=&\ -(8\sin(f)+12)\theta^{15}-(8\cos(f)+4)\theta^{16}-(12-8\cos(f))\theta^{25}-(8\sin(f)-4)\theta^{26}\\
    &+(4\cos(f)-4\sin(f)-8)(\theta^{38}-\theta^{47})+4(\cos(f)+\sin(f))(\theta^{37}+\theta^{48})\\
    T(E_7)=&\ -(8\cos(f)+4)\theta^{17}+(8\sin(f)-12)\theta^{18}-(8\sin(f)+4)\theta^{27}+(12-8\cos(f))\theta^{28}\\
    &+(4\cos(f)+4\sin(f)-8)(\theta^{46}+\theta_{35})+4(\cos(f)-\sin(f))(\theta^{36}-\theta^{45})\\
    T(E_8)=&\ -(8\sin(f)-12)\theta^{17}-(8\cos(f)+4)\theta^{18}+(8\cos(f)-12)\theta^{27}-(4+8\sin(f))\theta^{28}\\
    &-(4\cos(f)+4\sin(f)-8)(\theta^{36}-\theta^{45})+4(\cos(f)-\sin(f))(\theta^{35}-\theta^{46})
\end{align*}
    From this one finds that
\begin{equation}
    \mathrm{div}\ T = 32\sin(f) (\theta^{12}+\theta^{34}-\theta^{56}-\theta^{78}).
\end{equation}
When $f=0$, we see that $\Om$ indeed defines a harmonic QK structure, yet $T$ is non-zero. The flow equation (\ref{harmonicQKflowequation}) becomes the ODE
\[\frac{d}{dt}(\cos(f(t)))=128(1-\cos^2(f(t))),\]
where now we consider $f$ as a function of $t$ only. The general solution is given by
\[\cos(f(t))=\tanh(128t),\]
and hence $\displaystyle\lim_{t\to \infty}\cos(f(t)) = 1$.
Indeed we also see that the total energy
\[E(\Om(t))=3584\big(1-\frac{4}{7}\cos(f(t))\big)\vol(SU(3))\]
is decreasing. In future work we hope to study more systematically harmonic $\Sp(2)\Sp(1)$- and $\Spin(7)$-structures on $\SU(3)$, in a similar spirit to the study of harmonic homogeneous $\rG_2$-structures in \cite{Loubeau2022}. In particular, it would be interesting to see if the above harmonic structure is a global minimum of the energy functional.

\begin{remark}
    It is a well-known phenomenon in the context of Ricci flow that the symmetry group (in this case the isometry group) is preserved in finite time, but can increase in the long-time limit. Our example shows that a similar feature holds for the harmonic flow of QK structures: while $\Om(t)$ is merely $\SU(3) \times \mathrm{U}(1)^2$-invariant, the limit is actually $\SU(3) \times \mathrm{U}(2)$-invariant (see Proposition \ref{prop: rightinvariance}).
\end{remark}

\subsection{Steady soliton}
\label{sec: steadysolitonexample}

We shall now construct a steady soliton solution to (\ref{harmonicQKflowequation}) on $\R^8$ endowed with its standard Euclidean metric. Recall from Corollary \ref{cor: shrinkersandexpandeers} that only shrinking and expanding solitons have to be compatible with the Euclidean metric, we do not know if this also has to be the case for steady solitons as well.

Motivated by our examples above, we define a QK $4$-form $\Om$ by expression (\ref{qk4form}), substituting
\begin{align*}
    dx_1 &\mapsto \hspace{0.9em}\cos(e^{x_1})dx_1 +\sin(e^{x_1}) dx_2,\\
    dx_2 &\mapsto -\sin(e^{x_1})dx_1 +\cos(e^{x_1}) dx_2.
\end{align*}
    This simply corresponds to rotating the $1$-forms $dx_1$ and $dx_2$, so $\Om$ still induces the Euclidean metric $g_0$. Unlike in our previous examples however, observe that now the rotating function is non-constant and hence $\Om$ is not torsion-free, since for instance $d\Om \neq 0$. A simple computation shows that the torsion $T$ is given by:
\begin{align*}
    T(\partial_{x_1}) &= -8 e^{x_1}(dx_{12}+dx_{34}-dx_{56}-dx_{78}),\\
    T(\partial_{x_i}) &=0, \qforq  i=2,...,8.
\end{align*}
Since the Levi-Civita connection of $g_0$ is just the flat connection, it is easy to see that 
\[
    \mathrm{div}(T)= -8 e^{x_1}(dx_{12}+dx_{34}-dx_{56}-dx_{78})= T(\partial_{x_1}).
\]
Since $X=\partial_{x_1}$ is a gradient Killing vector field, it follows from (\ref{gradientsoliton}) that this corresponds to a steady gradient soliton. It would be interesting to see if similar methods can be used to find steady solitons on other manifolds that $\R^8$.

Surprisingly enough, a similar modification to the standard $\Spin(7)$ $4$-form (\ref{eq: spin7form}) yields a steady $\Spin(7)$ soliton cf. \cite{Dwivedi2021}. More precisely, using the same rotation as above for $dx_1$ and $dx_2$ in the expression (\ref{eq: spin7form}), one can repeat an analogous computation as above to show that the resulting $\Spin(7)$-structure is a soliton for the harmonic flow of $\Spin(7)$-structures. It seems plausible that a similar procedure can be adapted to work for $\rG_2$ and other cases as well, cf.  \cites{Dwivedi2019, Loubeau2019}.

\addcontentsline{toc}{section}{References}
\bibliography{Bibliografia-2021-10}

\begin{bibdiv}
\begin{biblist}

\bib{Besse2008}{misc}{
      author={Besse, Arthur~L.},
       title={{Einstein} manifolds. {R}eprint of the 1987 edition. {C}lassics
  in {M}athematics},
   publisher={Springer-Verlag, Berlin},
        date={2008},
}

\bib{ContiMadsen2015}{article}{
      author={Conti, Diego},
      author={Madsen, Thomas~Bruun},
       title={Harmonic structures and intrinsic torsion},
        date={2015},
        ISSN={1531-586X},
     journal={Transformation Groups},
      volume={20},
      number={3},
       pages={699\ndash 723},
         url={https://doi.org/10.1007/s00031-015-9325-x},
}

\bib{ChenStruwe1989}{article}{
      author={Chen, Yunmei},
      author={Struwe, Michael},
       title={Existence and partial regularity results for the heat flow for
  harmonic maps.},
        date={1989},
     journal={Mathematische Zeitschrift},
      volume={201},
      number={1},
       pages={83\ndash 104},
         url={http://eudml.org/doc/174037},
}

\bib{Dwivedi2019}{article}{
      author={Dwivedi, Shubham},
      author={Giannotis, Panagiotis},
      author={Karigiannis, Spiro},
       title={A gradient flow of isometric {${\rm G}_2$-structures}},
        date={2019},
      eprint={1904.10068},
}

\bib{Dadok1988}{article}{
      author={Dadok, Jiri},
      author={Harvey, Reese},
      author={Morgan, Frank},
       title={Calibrations on $\mathbb{R}^8$},
        date={1988},
        ISSN={00029947},
     journal={Transactions of the American Mathematical Society},
      volume={307},
      number={1},
       pages={1\ndash 40},
         url={http://www.jstor.org/stable/2000748},
}

\bib{Dwivedi2021}{article}{
      author={{Dwivedi}, Shubham},
      author={{Loubeau}, Eric},
      author={S{\'a}~Earp, Henrique~N.},
       title={Harmonic flow of $\mathrm{Spin}(7)$-structures},
        date={2021},
      eprint={2109.06340},
}

\bib{Fadel2022}{article}{
      author={Fadel, Daniel},
      author={Loubeau, Eric},
      author={Moreno, Andr{\'e}s},
      author={S{\'a}~Earp, Henrique~N.},
       title={Flows of geometric structures},
        date={2022},
      eprint={arxiv:2211.05197 [math.DG]},
}

\bib{UdhavFowdar4}{article}{
      author={{Fowdar}, Udhav},
       title={{Einstein metrics on bundles over hyperK{\"a}hler manifolds}},
        date={2021},
     journal={(to appear in Communications in Mathematical Physics)},
      eprint={2105.04254 [math.DG]},
}

\bib{FowdarSalamon2021}{article}{
      author={Fowdar, Udhav},
      author={Salamon, Simon},
       title={Symmetries, {t}ensors, and the {H}orrocks bundle},
        date={2022},
     journal={Differential Geometry and its Applications},
      volume={82},
       pages={101892},
}

\bib{GraysonHamilton1996}{article}{
      author={Grayson, Matthew},
      author={Hamilton, Richard~S.},
       title={The formation of singularities in the harmonic map heat flow},
        date={1996},
     journal={Communications in Analysis and Geometry},
      volume={4},
      number={4},
       pages={525\ndash 546},
}

\bib{Gibbons01}{article}{
      author={{Gibbons}, Garry~W.},
      author={{L{\"u}}, Hong},
      author={{Pope}, Christopher~N.},
      author={{Stelle}, Kellogg~S.},
       title={{Supersymmetric domain walls from metrics of special holonomy}},
        date={2002},
     journal={Nuclear Physics B},
      volume={623},
       pages={3\ndash 46},
}

\bib{Grigorian2019}{article}{
      author={Grigorian, Sergey},
       title={{Estimates and monotonicity for a heat flow of isometric ${\rm
  G_2}$-structures}},
        date={2019},
     journal={Calculus of Variations and Partial Differential Equations},
      volume={58},
       pages={175},
}

\bib{Hamilton1993}{article}{
      author={Hamilton, Richard~S.},
       title={Monotonicity formulas for parabolic flows on manifolds},
        date={1993},
     journal={Communications in Analysis and Geometry},
      volume={1},
      number={1},
       pages={127\ndash 137},
         url={https://doi.org/10.4310/cag.1993.v1.n1.a7},
}

\bib{Hamilton1995}{article}{
      author={Hamilton, Richard~S.},
       title={A compactness property for solutions of the {R}icci flow},
        date={1995},
     journal={American Journal of Mathematics},
      volume={117},
      number={3},
       pages={545\ndash 572},
         url={http://www.jstor.org/stable/2375080},
}

\bib{HeLi2021}{article}{
      author={He, Weiyong},
      author={Li, Bo},
       title={The harmonic heat flow of almost complex structures},
        date={2021},
     journal={Transactions of the American Mathematical Society},
      volume={374},
       pages={6179\ndash 6199},
}

\bib{Harvey1982}{article}{
      author={Harvey, F.~R.},
      author={Lawson, H.~B., Jr.},
       title={Calibrated geometries},
        date={1982},
        ISSN={0001-5962},
     journal={Acta Math.},
      volume={148},
       pages={47\ndash 157},
      review={\MR{MR666108 (85i:53058)}},
}

\bib{Joyce1992}{article}{
      author={Joyce, Dominic},
       title={{Compact hypercomplex and quaternionic manifolds}},
        date={1992},
     journal={Journal of Differential Geometry},
      volume={35},
      number={3},
       pages={743\ndash 761},
         url={https://doi.org/10.4310/jdg/1214448266},
}

\bib{Karigiannis2005}{article}{
      author={Karigiannis, Spiro},
       title={Deformations of {$G_2$} and {$Spin(7)$} structures},
        date={2005},
     journal={Canadian Journal of Mathematics},
      volume={57},
      number={5},
       pages={1012–1055},
}

\bib{Karigiannis2007}{article}{
      author={Karigiannis, Spiro},
       title={Flows of {$\rG_2$-structures}},
        date={2007},
     journal={The Quarterly Journal of Mathematics},
      volume={60},
      number={4},
}

\bib{Kobayashi1969a}{book}{
      author={Kobayashi, Shoshichi},
      author={Nomizu, Katsumi},
       title={Foundations of differential geometry},
   publisher={Interscience publishers New York},
        date={1969},
      volume={II},
}

\bib{Kuhnel1995}{article}{
      author={{K{\"u}hnel}, Wolfgang},
      author={{Rademacher}, Hans-Bert},
       title={{Conformal diffeomorphisms preserving the Ricci tensor}},
        date={1995},
     journal={Proceedings of the American Mathematical Society},
      volume={123},
      number={9},
       pages={2841\ndash 2848},
}

\bib{LeBrunInfiniteQK}{article}{
      author={LeBrun, {Claude}},
       title={On complete quaternionic-{K}{\"a}hler manifolds},
        date={1991},
     journal={Duke Mathematical Journal},
      volume={63},
      number={3},
       pages={723\ndash 743},
         url={https://doi.org/10.1215/S0012-7094-91-06331-3},
}

\bib{Loubeau2022}{article}{
      author={Loubeau, Eric},
      author={Moreno, Andr{\'e}s~J.},
      author={S{\'a}~Earp, Henrique~N.},
      author={Saavedra, Julieth},
       title={Harmonic {$Sp(2)$}-invariant {$G_2$}-structures on the
  $7$-sphere},
        date={2022},
        ISSN={1559-002X},
     journal={The Journal of Geometric Analysis},
      volume={32},
      number={9},
       pages={240},
         url={https://doi.org/10.1007/s12220-022-00953-9},
}

\bib{Loubeau2019}{article}{
      author={Loubeau, Eric},
      author={S{\'a}~Earp, Henrique~N.},
       title={Harmonic flow of geometric structures},
        date={2019},
      eprint={1907.06072},
}

\bib{PoonSalamon1991}{article}{
      author={Poon, Yat},
      author={Salamon, Simon},
       title={Quaternionic {K}{\"a}hler {$8$}-manifolds with positive scalar
  curvature},
        date={1991},
     journal={Journal of Differential Geometry},
      volume={33},
      number={2},
       pages={363\ndash 378},
         url={https://doi.org/10.4310/jdg/1214446322},
}

\bib{Salamon2001}{article}{
      author={Salamon, Simon},
       title={Almost parallel structures},
        date={2001},
        ISSN={0219-1997},
     journal={Communications in {C}ontemporary {M}athematics},
      volume={288},
       pages={162\ndash 181},
}

\bib{Salamon1982}{article}{
      author={Salamon, Simon},
       title={Quaternionic {K}{\"a}hler manifolds},
        date={1982},
        ISSN={1432-1297},
     journal={Inventiones mathematicae},
      volume={67},
      number={1},
       pages={143\ndash 171},
         url={https://doi.org/10.1007/BF01393378},
}

\bib{Salamon1989}{book}{
      author={Salamon, Simon},
       title={Riemannian geometry and holonomy groups},
      series={Pitman Research Notes in Mathematics Series},
   publisher={Longman Scientific \& Technical},
     address={Harlow},
        date={1989},
      volume={201},
        ISBN={0-582-01767-X},
}

\bib{Swann1991}{article}{
      author={Swann, Andrew},
       title={Hyper{K}{\"a}hler and quaternionic {K}{\"a}hler geometry},
        date={1991},
     journal={Mathematische Annalen},
      volume={289},
      number={3},
       pages={421\ndash 450},
         url={http://eudml.org/doc/164790},
}

\bib{Tashiro1965}{article}{
      author={Tashiro, Yoshihiro},
       title={Complete {R}iemannian manifolds and some vector fields},
        date={1965},
        ISSN={00029947},
     journal={Transactions of the American Mathematical Society},
      volume={117},
       pages={251\ndash 275},
         url={http://www.jstor.org/stable/1994206},
}

\end{biblist}
\end{bibdiv}

\end{document}